\documentclass[a4paper,10pt,reqno]{amsart}

\usepackage[english]{babel}
\usepackage{amsmath,amssymb,amsthm}
\usepackage{hyperref}
\usepackage[centering]{geometry}
\usepackage{xcolor}

\renewcommand{\epsilon}{\varepsilon}            

\newcommand{\R}{\ensuremath{\mathbb{R}}}

\newtheorem{theorem}{Theorem}[section]   
\newtheorem*{theorem*}{Theorem}          
\newtheorem{lemma}[theorem]{Lemma}
\newtheorem{proposition}[theorem]{Proposition}

\theoremstyle{definition}

\newtheorem{corollary}[theorem]{Corollary}
\newtheorem{example}[theorem]{Example}

\newtheorem*{remark}{Remark}

\numberwithin{equation}{section}

\title[]{Minimal Lagrangian submanifolds of the complex hyperquadric}
\thanks{}

\author{Haizhong Li \and Hui Ma \and Joeri Van der Veken \and Luc Vrancken \and Xianfeng Wang}

\address{H.~Li \and H.~Ma, Department of Mathematical Sciences, Tsinghua University, Beijing, 100084, P.R. China}
\email{lihz@tsinghua.edu.cn, ma-h@tsinghua.edu.cn}
\address{J. Van der Veken \and L.~Vrancken, KU\ Leuven, Department of Mathematics, Celestijnenlaan 200B -- Box 2400, 3001 Leuven, Belgium}
\email{joeri.vanderveken@kuleuven.be, luc.vrancken@kuleuven.be}
\address{L.~Vrancken, Universit\'e Polytechnique Hauts-de-France, Campus du Mont Houy, 59313 Valenciennes Cedex 9, France}
\email{luc.vrancken@uphf.fr}
\address{X. Wang, School of Mathematical Sciences and LPMC,
Nankai University,
Tianjin, 300071,  P. R. China; Mathematical Sciences Institute,
Australian National University, Canberra,
ACT 2601 Australia}
\email{wangxianfeng@nankai.edu.cn, xianfeng.wang@anu.edu.au}


\date{}

\begin{document}

\begin{abstract}
We introduce a structural approach to study Lagrangian submanifolds of the complex hyperquadric in arbitrary dimension by using its family of non-integrable almost product structures. In particular, we define local angle functions encoding the geometry of the Lagrangian submanifold at hand. We prove that these functions are constant in the special case that the Lagrangian immersion is the Gauss map of an isoparametric hypersurface of a sphere and give the relation with the constant principal curvatures of the hypersurface. We also use our techniques to classify all minimal Lagrangian submanifolds of the complex hyperquadric which have constant sectional curvatures and all minimal Lagrangian submanifolds for which all, respectively all but one, local angle functions coincide.
\end{abstract}

\maketitle

\section{Introduction}

In this paper, we investigate the geometry of Lagrangian submanifolds of the  complex hyperquadric $Q^n$, which is a homogeneous complex $n$-dimensional K\"ahler manifold. The study of Lagrangian submanifolds originates from symplectic geometry and classical mechanics. Let $(N,\omega)$ be a $2n$-dimensional symplectic manifold with a symplectic form $\omega$, we call a submanifold $f:M\to(N,\omega)$ Lagrangian, if $f^{*}\omega=0$ and the dimension of $M$ is half the dimension of $N$. In particular, if $N$ is a K\"ahler manifold, then $N$ admits complex, Riemannian and symplectic structures which are compatible with each other, and the condition $f^{*}\omega=0$ is equivalent to the complex structure $J$ of $N$ interchanging the tangent and the normal spaces. The study of Lagrangian submanifolds of K\"ahler manifolds is a classic topic and was initiated in the 1970's by Chen and Ogiue~\cite{CO1974}. For a review on Riemannian geometry of Lagrangian submanifolds we refer to~\cite{Chen2001,Chen2011} and the references therein. The simplest examples of K\"ahler manifolds are complex space forms which have constant holomorphic sectional curvatures, and the geometry of Lagrangian submanifolds of complex space forms have been widely studied and well understood in some sense. Meanwhile, Lagrangian submanifolds of other K\"ahler manifolds have not been deeply understood.

The complex hyperquadric $Q^n$ is a  compact complex hypersurface of the complex projective space $\mathbb C P^{n+1}$ defined by the homogeneous quadratic equation $z_0^2 +z_1^2 + \ldots + z_{n+1}^2 = 0$. It can be identified with the Grassmann manifold  of oriented $2$-planes, is a compact Hermitian symmetric space of rank $2$ and provides a very good example of a K\"ahler-Einstein manifold. There is a fundamental fact which concerns the relation between Lagrangian geometry of the complex hyperquadric and hypersurface geometry of the unit sphere, that is, the Gauss map of any oriented hypersurface of the unit sphere $S^{n+1}$ is always a Lagrangian submanifold  of the complex hyperquadric $Q^n$.  In a special case, given an isoparametric hypersurface of the unit sphere, one can get a minimal Lagrangian submanifold of the complex hyperquadric by using its Gauss map. The minimality can be proved by applying Palmer's nice formula involving the mean curvature form of the Lagrangian submanifold of the complex hyperquadric and the principal curvatures of the hypersurface of the unit sphere (see \cite{Palmer1994,Palmer}). The geometry of Lagrangian submanifolds of $Q^n$ obtained by the Gauss map of isoparametric hypersurfaces of the unit sphere have been systematically studied from the point of view of Lie theory and Hamiltonian deformation theory by the second author and Ohnita (see \cite{MaOhnita1,MaOhnita2,MaOhnita3}), where they obtain a classification of all compact homogeneous Lagrangian submanifolds of the complex hyperquadric and determine the Hamiltonian stability for the Gauss maps of all the known homogeneous isoparametric hypersurfaces. Notwithstanding, the geometry of Lagrangian submanifolds of the complex hyperquadric is far from being well understood, especially from the geometric point of view. This motivates our study and the aim of this paper is to understand the geometry of Lagrangian submanifold of the complex hyperquadric in a more geometric way.

It is well known that $Q^n$ carries a family of almost product structures, see for example \cite{R1}. In this paper, we introduce a new structural approach to study Lagrangian submanifolds of $Q^n$. By using one of the almost product structures, we define two symmetric operators and $n$ local \textit{angle functions} on every such Lagrangian submanifold. It turns out that these angle functions have very nice relations with the second fundamental form and can determine most of the geometry of the Lagrangian submanifold. By use of our new approach, we obtain a correspondence theorem between minimal Lagrangian submanifolds of $Q^n$ with constant angle functions and the Gauss maps of isoparametric hypersurfaces of $S^{n+1}$.
\begin{theorem} \label{theo1}
	Let $a: M^n \to S^{n+1}(1)$ be an isoparametric hypersurface with unit normal $b$ and principal curvatures $\lambda_1,\ldots,\lambda_n$. Then the Gauss map $G: M^n \to Q^n$ is a minimal Lagrangian immersion and the difference between any two local angle functions is constant. Moreover,
	\begin{itemize}
		\item [(i)]if the almost product structure $A \in \mathcal A$ is chosen as in Example \ref{ex1} for the canonical horizontal lift $\hat G$ given in \eqref{def_Ghat}, then all the angle functions $\theta_1,\ldots,\theta_n$ are constant and, when put in the right order, they are given by
		\begin{equation} \label{relation_lambdaj_thetaj_1}
			\lambda_j = \cot \theta_j;
		\end{equation}
		\item [(ii)] if the almost product structure $A \in \mathcal A$ is chosen as in Example \ref{ex2}, then again all the angle functions are constant and, when put in the right order,
		\begin{equation} \label{relation_lambdaj_thetaj_2}
			\lambda_j = \cot(\theta_j+c)
		\end{equation}
		for a real constant $c$ which is independent of the index $j$.
	\end{itemize}
	Conversely, consider a minimal Lagrangian immersion $f: M^n \to Q^n$ of a simply connected manifold with constant angle functions $\theta_1,\ldots,\theta_n$. Then for every real constant $c$ with $\sin(\theta_j+c) \neq 0$ for all $j=1,\ldots,n$, there is an isoparametric immersion $M^n \to S^{n+1}(1)$ with Gauss map $f$, whose principal curvatures are given by \eqref{relation_lambdaj_thetaj_2}.
\end{theorem}

By applying Theorem \ref{theo1}, we can find all minimal Lagrangian submanifolds with constant local angle functions starting from isoparametric hypersurfaces of the unit sphere, see Corollary \ref{cor2}. We also obtain the classification of Lagrangian submanifolds which are totally geodesic (Theorem \ref{theoTG}) and the classification of  minimal Lagrangian submanifolds for which all local angle functions are the same (Theorem \ref{theoSameAngles}). Although Theorem \ref{theo1} is stated only in the case of constant angle functions, from the proof of this theorem, one can get similar local correspondence between a general minimal Lagrangian submanifold of $Q^n$ and the Gauss map of a hypersurface of $S^{n+1}$, with the same relation between the angle functions of the Lagrangian submanifold and the principal curvatures of the hypersurface, while the angle functions and the principal curvatures  are not necessarily constant.

Note that the classification of  minimal Lagrangian submanifolds of constant  sectional curvature in complex space forms is a classic result proved  by Ejiri \cite{Ejiri1982}, using essentially the Gauss and Codazzi equations.  In the case of the complex hyperquadric $Q^n$, it is impossible to get a similar classification following Ejiri's method, as the curvature tensor of $Q^n$ is much more complicated than that of a complex space form and the Gauss equations become very difficult to solve directly. Therefore, we use a quite different approach, by making full use of the angle functions. We obtain the following classification theorem.

\begin{theorem} \label{theoCSCconclusion}
	let $f:M^n \to Q^n$, $n \geq 2$, be a minimal Lagrangian immersion such that $M^n$ has constant sectional curvature $c$. Then $f$ is one of the following:
	\begin{itemize}
		\item[(i)] $f$ is the Gauss map of a part of the standard embedding $S^n(r) \to S^{n+1}(1)$;
		\item[(ii)] $n=2$ and $f$ the Gauss map of a part of the standard embedding $S^1(r_1) \times S^1(r_2) \to S^3(1)$;
		\item[(iii)] $n=3$ and $f$ is the Gauss map of a part of a tube around the Veronese surface in $S^4(1)$.
	\end{itemize}
	The constant sectional curvatures in these three cases are $c=2$, $c=0$ and $c=1/8$ respectively.
\end{theorem}

As another successful application of our new techniques, we classify all minimal Lagrangian submanifolds for which all but one local angle functions coincide.
\begin{theorem} \label{theoAllButOneSameAngles}
	Let $f:M^n \to Q^n~(n\geq3)$ be a minimal Lagrangian submanifold  for which $n-1$ local angle functions are the same. If $A \in \mathcal A$ is chosen as in Example \ref{ex2}, then $\theta_1=(n-1)\alpha \mod\pi$ and $\theta_2=\cdots=\theta_n=-\alpha \mod\pi$ for some  local function $\alpha$.
	\begin{itemize}
		\item[(i)] If $\alpha=0\mod\pi$, then  $f$ is the Gauss map of a part of the standard embedding $S^{n}(r)\to S^{n+1}(1)$.
		\item[(ii)] If $\alpha$ is a non-zero constant modulo $\pi$, then  $f$ is the Gauss map of a part of the standard embedding $S^{1}(r_1)\times S^{n-1}(r_2)\to S^{n+1}(1)$.
		\item[(iii)] If $\alpha$ is not a constant modulo $\pi$, then  $M^n$ must be a warped product $I\times_\rho S^{n-1}(1)$ with $\rho(\alpha)=|c_1(\sin n\alpha)^{-\frac{1}{n}}|$ for some positive constant $c_1$, and the angle function $\alpha$ satisfies the following first order ordinary differential equation:
		\begin{equation}\label{6.1}
			(c_1(\sin (n\alpha))^{-\frac{1}{n}})^2(2+(\frac{d \alpha}{ds})^2(\sin{n\alpha})^{-2})=1,	\end{equation}
		where $\frac{d}{d s}=e_1$ is the tangent vector on the base $I$ of the warped product $ I\times_\rho S^{n-1}(1)$.
		Moreover, $f$ is locally isometric to the Gauss map of a rotational hypersurface of $S^{n+1}(1)$ with the profile curve $\gamma(\theta)\subset S^2(1)$ given by
		\begin{equation}\label{6.2}
			\gamma(\theta)=\big(-\sin{\alpha}\sqrt{1-(\frac{d\alpha}{d\theta})^2},\cos{\alpha}\sin\theta-\sin{\alpha}\cos{\theta} \frac{d\alpha}{d\theta}, -\cos{\alpha}\cos\theta-\sin{\alpha}\sin{\theta} \frac{d\alpha}{d\theta} \big),
		\end{equation}
		where $\alpha$ is the angle function and satisfies the following second order ordinary differential equation:
		\begin{equation}\label{6.3}
			\frac{d^2\alpha}{d\theta^2}=(1-(\frac{d\alpha}{d\theta})^2)\cot{(n\alpha)},~|\frac{d\alpha}{d\theta}|<1.
		\end{equation}
	\end{itemize}
\end{theorem}
As proposed in \cite{siffert}, one can expect that our approach via Lagrangian submanifolds of $Q^n$, will provide some new understanding of the isoparametric theory in the unit sphere.  For more results about isoparametric theory and its applications, we refer to \cite{GT}, \cite{QT}, \cite{TY}
and the references therein.

The paper is organized as follows. In Section 2, we review some basic definitions and properties of the complex hyperquadric $Q^n$. In Section 3, starting from an almost product structure on $Q^n$, we first define two endomorphisms on the tangent spaces of a Lagrangian submanifold $M^n$ of $Q^n$, and then introduce $n$ local angle functions on $M$.  We give two examples to show that we can choose appropriate almost product structures to make the angle functions having special properties. We prove some relations  among the angle functions, the second fundamental form and the Levi-Civita connection form. We also give the Gauss and Codazzi equations for a Lagrangian submanifold of $Q^n$. In Section 4, we study the correspondence between Lagrangian submanifolds of $Q^n$ and hypersurfaces of $S^{n+1}$, and prove Theorem \ref{theo1}. Section 5 is devoted to the proof of Theorem \ref{theoCSCconclusion}, where Lemma \ref{lem5.1} is the key  step, which gives us more information about the second fundamental form and the angle functions, under the assumption of constant sectional curvature. In Section 6, we classify all minimal Lagrangian submanifolds with $n-1$ equal  local angle functions by proving Theorem  \ref{theoAllButOneSameAngles}.

\textbf{Acknowledgments:} This research was supported by the Tsinghua University - KU Leuven Bilateral Scientic Cooperation Fund and by the NSFC-FWO grant No.11961131001. H.~Li is supported by NSFC No.11831005 and No.11671224. H.~Ma is supported by NSFC No.11831005 and No.11671223. J. Van der Veken and L. Vrancken are supported by project 3E160361 of the KU Leuven Research Fund. J. Van der Veken is supported by the Excellence Of Science project G0H4518N of the Belgian government. X. Wang is supported by NSFC No.11571185 and the Fundamental Research Funds for the Central Universities, and she would also like to express her deep gratitude to the Mathematical Sciences Institute at the Australian National University for its hospitality and to Prof.  Ben Andrews for his encouragement and help during her stay in MSI of ANU as a Visiting Fellow, while part of this work was completed. The authors would also like to thank the referees for carefully reading of this paper and providing some helpful suggestions.

\section{The geometry of the complex hyperquadric}

Let $\mathbb{C} P^{n+1}(4)$ be the complex projective space of complex dimension $n+1$ equipped with the Fubini-Study metric $g_{FS}$ of constant holomorphic sectional curvature $4$. Then the Hopf fibration
\begin{equation} \label{def_pi}
	\pi : S^{2n+3}(1) \subset \mathbb{C}^{n+2} \to \mathbb{C} P^{n+1}(4): z \mapsto [z]
\end{equation}
is a Riemannian submersion from the unit sphere of real dimension $2n+3$ to $\mathbb{C} P^{n+1}(4)$. Remark that for any $z \in S^{2n+3}(1)$ we have $\pi^{-1}\{[z]\} = \{e^{\sqrt{-1}t}z \ | \ t\in\R\}$ and $\ker (d\pi)_z = \mathrm{span} \{\sqrt{-1}z\}$. The complex structure $J$ on $\mathbb{C} P^{n+1}(4)$ is induced from multiplication by $\sqrt{-1}$ on $TS^{2n+3}(1)$ and it is well-known that $(\mathbb{C} P^{n+1}(4), g_{FS}, J)$ is a K\"ahler manifold.

We define the \textit{complex hyperquadric} of complex dimension $n$ as the following complex hypersurface of $\mathbb{C} P^{n+1}(4)$:
\begin{equation} \label{def_Qn}
	Q^n = \{ [(z_0,z_1,\ldots,z_{n+1})] \in \mathbb{C} P^{n+1}(4) \ | \ z_0^2 +z_1^2+ \cdots + z_{n+1}^2 = 0 \}.
\end{equation}
If $Q^n$ is equipped with the induced metric $g_{FS}|_{Q^n}$, which we will denote by $g$, and the induced almost complex structure $J|_{Q^n}$, which we will again denote by $J$, then $(Q^n,g,J)$ is of course a K\"ahler manifold itself. The inverse image of $Q^n$ under the Hopf fibration is the Stiefel manifold
\begin{equation}
	V_2(\R^{n+2}) = \left\{ u+\sqrt{-1}v \ \left| \ u,v\in\R^{n+2}, \ \langle u,u \rangle = \langle v,v \rangle = \frac 12, \ \langle u,v \rangle = 0 \right. \right\} \subset S^{2n+3}(1)
\end{equation}
of real dimension $2n+1$, where $\langle \cdot,\cdot \rangle$ denotes the Euclidean inner product on $\R^{n+2}$.

The normal space to $V_2(\R^{n+2})$ in $S^{2n+3}(1)$ at a point $z$ is spanned by $\bar z$ and $\sqrt{-1}\bar z$, which implies that the normal space to $Q^n$ in $\mathbb{C} P^{n+1}(4)$ at a point $[z]$ is spanned by $(d\pi)_z(\bar z)$ and $J(d\pi)_z(\bar z)=(d\pi)_z(\sqrt{-1}\bar z)$, where $z$ is any representative of $[z]$. Remark that these vectors depend on the chosen representative $z$. We denote by $\mathcal A$ the set of all shape operators of $Q^n$ in $\mathbb{C} P^{n+1}(4)$ associated with unit normal vector fields. Then $\mathcal A$ is a collection of $(1,1)$-tensor fields on $Q^n$ and one can deduce the following (see for example \cite{R1} or \cite{Smyth1967})).

\begin{lemma} \label{lem1}
	Any $A \in \mathcal A$ satisfies
	\begin{itemize}
		\item[(i)] $A^2=\text{Id}$,  i.e.,  $A$ is involutive.
		\item[(ii)] $A$ is symmetric.
		\item[(iii)] $A$ anti-commutes with $J$.
	\end{itemize}
\end{lemma}

Properties (i) and (ii) in Lemma \ref{lem1} are equivalent to saying that $\mathcal A$ is a family of almost product structures on $Q^n$. However, these almost product structures are not always integrable. In fact, we have the following result.

\begin{lemma} [\cite{Smyth1967}]\label{lem2}
	Let $\xi$ be a unit normal vector field along $Q^n$ in $\mathbb{C} P^{n+1}(4)$ with corresponding shape operator $A \in \mathcal A$. Then there exists a non-zero one-form $s$ such that
	\begin{align}
		& \nabla^{\mathbb{C} P^{n+1}(4)}_X \xi = - AX + s(X) J\xi, \label{def_s}\\
		& \nabla^{Q^n}_X A = s(X) JA \label{nabla_A}
	\end{align}
	for all $X$ tangent to $Q^n$, where $\nabla^{\mathbb{C} P^{n+1}(4)}$ and $\nabla^{Q^n}$ are the Levi Civita connections of $\mathbb{C} P^{n+1}(4)$ and $Q^n$ respectively.
\end{lemma}

The equation of Gauss for $Q^n$ as a submanifold of $\mathbb{C} P^{n+1}(4)$ yields the following expression for the Riemannian curvature tensor of $Q^n$:
\begin{align} \label{R_Qn}
	R^{Q^n}(X,Y)Z = \ & g(Y,Z)X - g(X,Z)Y + g(X,JZ)JY - g(Y,JZ)JX + 2g(X,JY)JZ \\
	& + g(AY,Z)AX - g(AX,Z)AY + g(JAY,Z)JAX - g(JAX,Z)JAY, \nonumber
\end{align}
where $A$ is any element of $\mathcal A$ and $X,~Y,~Z\in TQ^n$.  We can calculate straightforwardly from \eqref{R_Qn} that $Q^n$ is a K\"ahler-Einstein manifold with Einstein constant $2n$.

\section{Lagrangian submanifolds of the complex hyperquadric}

In the following sections, we consider an immersion $f: M^n \to Q^n$ of a manifold of real dimension $n$ into the complex hyperquadric of complex dimension $n$. If no confusion is possible, we will identify $M^n$ with its image and $(df)_p(T_pM^n)$ with $T_pM^n$ for every $p \in M^n$. Moreover, we will denote the metric on $M^n$ induced from the metric $g$ on $Q^n$, constructed above, again by $g$. As usual in complex geometry, we say that $f$ is \textit{Lagrangian} if $J$ maps the tangent space to $M^n$ at any point into the normal space to $M^n$ at that point and vice versa.

Fixing an almost product structure $A \in \mathcal A$ on $Q^n$, we can define at any point $p$ of a Lagrangian submanifold $M^n$ of $Q^n$ two endomorphisms $B$ and $C$ of $T_pM^n$ by putting
\begin{equation} \label{def_BC}
	AX = BX - JCX
\end{equation}
for all $X \in T_pM$, i.e., $BX$ is the component of $AX$ tangent to $M^n$ and $CX$ is the image under $J$ of the component of $AX$ normal to $M^n$. With these definitions, we have the following.

\begin{lemma} \label{lem3}
	$B$ and $C$ are symmetric endomorphisms of $T_pM^n$ which commute and satisfy $B^2+C^2=\mathrm{Id}$.
\end{lemma}

\begin{proof}
	Since $g(BX,Y)=g(AX,Y)$ and $g(CX,Y)=g(JAX,Y)$ for all $X,Y \in T_pM^n$, the endomorphisms $B$ and $C$ are symmetric because $A$ and $JA$ are symmetric.
	
	Furthermore, we have $X = A^2X = A(BX-JCX) = (B^2+C^2)X + J(BC-CB)X$ for an arbitrary $X \in T_pM^n$. Since the first term on the right hand side is tangent to $M^n$ and the second term on the right hand side is normal to $M^n$, we must have $(B^2+C^2)X=X$ and $(BC-CB)X=0$, which proves the result.
\end{proof}

Lemma \ref{lem3} implies that $B$ and $C$ are simultaneously diagonalizable and that the sum of the squares of corresponding eigenvalues must be $1$. Therefore, there exist an orthonormal basis $\{e_1,\ldots,e_n\}$ of $T_pM^n$ and real numbers $\theta_1,\ldots,\theta_n$, defined up to an integer multiple of $\pi$, such that
\begin{equation} \label{BC}
	Be_j = \cos(2\theta_j) e_j, \qquad Ce_j = \sin(2\theta_j) e_j
\end{equation}
for $j=1,\ldots,n$. The factor $2$ in front of the angles is just a choice for convenience, as it will simplify some of the expressions in the sequel.  We can rewrite \eqref{BC} as $Ae_j = \cos(2\theta_j) e_j - \sin(2\theta_j) Je_j$.

Working locally, we can regard $B$ and $C$ as symmetric $(1,1)$-tensor fields on $M^n$ which define a local orthonormal frame $\{e_1,\ldots,e_n\}$ and local angle functions $\theta_1,\ldots,\theta_n$ in a similar way as above. In general, these functions cannot be extended to global functions on $M^n$ and they are only determined up to an integer  multiple of $\pi$.

The following result states that changing $A \in \mathcal A$ will change the angle functions $\theta_1,\ldots,\theta_n$, but not the orthonormal frame $\{e_1,\ldots,e_n\}$.
\begin{lemma}\label{lem4}
	Let $f: M^n \to Q^n$ be a Lagrangian immersion and $A_0,A \in \mathcal A$. Then there exists a function $\varphi: M^n \to \R$ such that, along the image of $f$,
	\begin{equation}
		A = \cos\varphi \, A_0 + \sin\varphi \, JA_0.
	\end{equation}
	If $\{e_1,\ldots,e_n\}$ is a local orthonormal frame such that $A_0e_j = \cos(2\theta^0_j) e_j - \sin(2\theta^0_j) Je_j$ for $j=1,\ldots,n$, then $Ae_j = \cos(2\theta_j) e_j - \sin(2\theta_j) Je_j$ for $j=1,\ldots,n$, with
	\begin{equation} \label{relation_theta_j}
		\theta_j = \theta^0_j - \frac{\varphi}{2}.
	\end{equation}
\end{lemma}

\begin{proof}
	Assume that $A_0$ and $A$ are the shape operators associated with unit normal vector fields $\xi_0$ and $\xi$ respectively. Since $Q^n$ is a K\"ahler submanifold of $\mathbb{C} P^{n+1}(4)$, there is a function $\varphi:M^n \to \R$ such that, at every point of $M^n$, $\xi = \cos\varphi \, \xi_0 + \sin\varphi \, J\xi_0$, which implies that $A = \cos\varphi \, A_0 + \sin\varphi \, JA_0$ along the image of $f$. Now assume that $A_0e_j = \cos(2\theta^0_j) e_j - \sin(2\theta^0_j) Je_j$ for $j=1,\ldots,n$. Then it follows from a straightforward computation that $Ae_j = \cos(2\theta^0_j-\varphi) e_j - \sin(2\theta^0_j-\varphi) Je_j$ for $j=1,\ldots,n$.
\end{proof}

There are a few possible choices for the almost product structure $A \in \mathcal A$ on $Q^n$ which are adapted to a given Lagrangian submanifold $f: M^n \to Q^n$. We present two of them in the next examples.

\begin{example} \label{ex1}
	Assume that, apart from a Lagrangian immersion $f: M^n \to Q^n$, also a horizontal lift $\hat f: M^n \to V_2(\R^{n+2})$ of $f$ is given. Remark that it follows from \cite{Reckziegel} that any Lagrangian immersion into $Q^n$ locally allows such a horizontal lift. If $M^n$ is simply connected, the horizontal lift can be defined globally. Since the normal space to $V_2(\R^{n+2})$ in $S^{2n+3}(1) \subset \mathbb{C}^{n+2}$ at a point $z$ is the complex span of $\bar z$, one can take
	$$\xi_{f(p)}=(d\pi)_{\hat f(p)}\left(\overline{\hat{f}(p)}\right)$$
	as a unit normal vector field to $Q^n$ along the image of $f$,  and the corresponding shape operator is given by
	\begin{equation} \label{A0}
		AX = -(d\pi)\left(\overline{\hat X}\right),
	\end{equation}
	where $X$ is any vector tangent to $Q^n$ at a point $f(p)$ and $\hat X$ is its horizontal lift to $\hat f(p)$.
\end{example}

\begin{example} \label{ex2}
	Given a Lagrangian immersion $f: M^n \to Q^n$, one can choose $A \in \mathcal A$ such that the associated local angle functions satisfy
	\begin{equation}\label{sumangleszero}
		\theta_1 + \cdots + \theta_n = 0 \mod \pi.
	\end{equation}
	Indeed, let $A_0 \in \mathcal A$ be an arbitrary almost product structure with associated local angle functions $\theta_1^0, \ldots , \theta_n^0$ and put $\varphi= 2(\theta^0_1 + \cdots + \theta^0_n)/n$. If we choose $A \in \mathcal A$ such that $A = \cos\varphi \, A_0 + \sin\varphi \, JA_0$ along the image of $f$, then it follows from \eqref{relation_theta_j} that the local angle functions associated with $A$ satisfy \eqref{sumangleszero}. Remark that we will always work modulo $\pi$ for local angle functions, since they are only defined up to an integer multiple of $\pi$.

\end{example}

\begin{remark}
	The choice of $\varphi$, and hence of $A \in \mathcal A$, in Example \ref{ex2} is not uniquely determined. Indeed, for any $k \in \{0,\ldots,n-1\}$, the function $\varphi= 2(\theta^0_1 + \cdots + \theta^0_n)/n + 2k\pi/n$ gives rise to a different $A \in \mathcal A$ for which the angle functions satisfy \eqref{sumangleszero}.
\end{remark}

Let  $h$ be the second fundamental form of the Lagrangian immersion $f: M^n \to Q^n$, we define
\begin{equation}
	h_{ij}^k = g(h(e_i,e_j),Je_k)
\end{equation}
for all $i,j,k=1,\ldots,n$, to be the components of $h$. A fundamental property of Lagrangian submanifolds of K\"ahler manifolds implies that the components $h_{ij}^k$ are totally symmetric in the three indices (cf. \cite{Oh1990,Oh1993}). Furthermore, let $\nabla$ denote the induced connection on $M^n$ from the Levi Civita connection $\nabla^{Q^n}$ of $(Q^n,g)$, we define its connection forms by
\begin{equation}
	\omega_j^k(X) = g(\nabla_X e_j,e_k)
\end{equation}
for all $j,k=1,\ldots,n$ and all $X$ tangent to $M^n$. Remark that this family of one-forms is anti-symmetric in the indices. The following proposition relates the angle functions, the components of the second fundamental form and the connection forms.

\begin{proposition} \label{prop1}
	Let $M^n$ be a Lagrangian submanifold of $Q^n$ and assume that an almost product structure $A \in \mathcal A$ on $Q^n$ is fixed. Let $\{e_1,\ldots,e_n\}$ be a local orthonormal frame on $M^n$ constructed as above, then the following relations among the angle functions, the components of the second fundamental form and the connection forms hold:
	\begin{align}
		& e_i(\theta_j) = h_{jj}^i - \frac{s(e_i)}{2}, \label{intcond1} \\
		& \sin(\theta_j-\theta_k) \omega_j^k(e_i) = \cos(\theta_j-\theta_k) h_{ij}^k, \label{intcond2}
	\end{align}
	for all $i,j,k=1,\ldots,n$, with $j \neq k$. Here, $s$ is the one-form associated with $A$ as in Lemma \ref{lem2}.
\end{proposition}

\begin{proof}
	Combining the splitting $AX=BX-JCX$ with formula \eqref{nabla_A}, yields
	\begin{align}
		& (\nabla_X B)Y = s(X) CY + Jh(X,CY) + CJh(X,Y),\label{nablab} \\
		& (\nabla_X C)Y = -s(X) BY - Jh(X,BY) - BJh(X,Y)\label{nablac}
	\end{align}
	for all vector fields $X$ and $Y$ tangent to $M^n$. Evaluating these expressions for $X=e_i$ and $Y=e_j$ gives us two equalities between vectors.
	
	Comparing the components in $e_j$ gives respectively
	\begin{align*}
		& -2 \sin(2\theta_j)e_i(\theta_j) = s(e_i) \sin(2\theta_j) - 2\sin(2\theta_j)h_{ij}^j, \\
		& 2 \cos(2\theta_j)e_i(\theta_j) = -s(e_i) \cos(2\theta_j) + 2\cos(2\theta_j)h_{ij}^j.
	\end{align*}
	Since either $\sin(2\theta_j) \neq 0$ or $\cos(2\theta_j) \neq 0$, we conclude \eqref{intcond1}.
	
	On the other hand, comparing the components in $e_k$ for some $k \neq j$ gives respectively
	\begin{align*}
		& (\cos(2\theta_j)-\cos(2\theta_k)) \omega_j^k(e_i) = -(\sin(2\theta_j)+\sin(2\theta_k)) h_{ij}^k, \\
		& (\sin(2\theta_j)-\sin(2\theta_k)) \omega_j^k(e_i) = (\cos(2\theta_j)+\cos(2\theta_k)) h_{ij}^k,
	\end{align*}
	or, equivalently,
	\begin{align*}
		& -\sin(\theta_j+\theta_k)\sin(\theta_j-\theta_k) \omega_j^k(e_i) = -\sin(\theta_j+\theta_k)\cos(\theta_j-\theta_k) h_{ij}^k, \\
		& \cos(\theta_j+\theta_k)\sin(\theta_j-\theta_k) \omega_j^k(e_i) = \cos(\theta_j+\theta_k)\cos(\theta_j-\theta_k) h_{ij}^k.
	\end{align*}
	Since either $\sin(\theta_j+\theta_k) \neq 0$ or $\cos(\theta_j+\theta_k) \neq 0$, we conclude \eqref{intcond2}.
\end{proof}

\begin{corollary} \label{cor1}
	Let $f: M^n \to Q^n$ be a minimal Lagrangian immersion for which the sum of the local angle functions is constant. This can for example be achieved by choosing $A \in \mathcal A$ as in Example \ref{ex2}. Then the one-form $s$ associated with $A$ vanishes on tangent bundle of $M^n$. In particular, for all $X$ tangent to $M$, one has $\nabla^{Q^n}_X A=0$ and, if $A$ is the shape-operator associated with a normal vector field $\xi$ along $Q^n$, also $\nabla^{\perp}_X \xi = 0$, where $\nabla^{\perp}$ is the normal connection of $Q^n$ in $\mathbb{C} P^{n+1}(4)$.
\end{corollary}

\begin{proof}
	Denote the angle functions by $\theta_1,\ldots,\theta_n$. Then
	\begin{equation} \label{phi_const}
		0 = e_i \left( \theta_1 + \cdots + \theta_n \right) = h_{11}^i + \cdots + h_{nn}^i - n \frac{s(e_i)}{2} = - n \frac{s(e_i)}{2}
	\end{equation}
	for any $i=1,\ldots,n$, where we used \eqref{intcond1} and the fact that $h_{11}^i + \cdots + h_{nn}^i$ is $n$ times the $Je_i$-component of the mean curvature vector. Hence $s(e_i)=0$ for $i=1,\ldots,n$ and thus $s$ vanishes on tangent bundle of $M^n$. The rest of the statement now follows directly from Lemma \ref{lem2}.
\end{proof}

To end this section, we state the equations of Gauss and Codazzi for a Lagrangian submanifold of $Q^n$.

\begin{proposition}[Equations of Gauss and Codazzi]
	Let $f: M^n \to Q^n$ be a Lagrangian immersion with second fundamental form $h$. Define $B$ and $C$ as above for any choice of $A \in \mathcal A$. Finally, denote by $R$ the Riemannian curvature tensor  of $M^n$ and by $\overline\nabla$ the Levi-Civita connection. Then
	\begin{equation} \label{gauss}
		\begin{aligned}
			g(R(X,Y)Z,W) = \ & g(Y,Z)g(X,W) - g(X,Z)g(Y,W) \\
			& + g(BY,Z) g(BX,W) - g(BX,Z) g(BY,W) \\
			& + g(CY,Z) g(CX,W) - g(CX,Z) g(CY,W) \\
			& + g(h(Y,Z),h(X,W)) - g(h(X,Z),h(Y,W))
		\end{aligned}
	\end{equation}
	and
	\begin{equation} \label{codazzi}
		\begin{aligned}
			(\overline\nabla h)(X,Y,Z) - (\overline\nabla h)(Y,X,Z) = \ & g(CY,Z) JBX - g(CX,Z) JBY \\
			& - g(BY,Z) JCX + g(BX,Z) JCY
		\end{aligned}
	\end{equation}
	for any vector fields $X$, $Y$, $Z$ and $W$ tangent to $M^n$.
\end{proposition}

\begin{proof}
	These follow immediately from the general forms of the equations of Gauss and Codazzi,
	\begin{align*}
		& g(R(X,Y)Z,W) = g(R^{Q^n}(X,Y)Z,W) + g(h(Y,Z),h(X,W)) - g(h(X,Z),h(Y,W)), \\
		& (\overline\nabla h)(X,Y,Z) - (\overline\nabla h)(Y,X,Z) = (R^{Q^n}(X,Y)Z)^{\perp},
	\end{align*}
	where the superscript $\perp$ denotes the component normal to $M^n$, by using \eqref{R_Qn} and \eqref{def_BC}.
\end{proof}

\begin{remark}
	If $\{e_1,\ldots,e_n\}$ is the local orthonormal frame constructed above and $\theta_1,\ldots,\theta_n$ are the local angle functions, then it follows from \eqref{gauss} and \eqref{intcond1} that the sectional curvature of the plane spanned by $e_i$ and $e_j$ is given by
	\begin{equation} \label{seccurv}
		\begin{aligned}
			K_{ij} &= g(R(e_i,e_j)e_j,e_i) \\
			&= 2 \cos^2(\theta_i-\theta_j) + g(h(e_i,e_i),h(e_j,e_j)) - g(h(e_i,e_j),h(e_i,e_j)) \\
			&= 2 \cos^2(\theta_i-\theta_j) + \sum_{k=1}^n \left( h_{ii}^k h_{jj}^k - (h_{ij}^k)^2 \right) \\
			&= 2 \cos^2(\theta_i-\theta_j) + \sum_{k=1}^n \left( \left( e_k(\theta_i)+\frac{s(e_k)}{2} \right) \left(e_k(\theta_j)+\frac{s(e_k)}{2}\right)-(h_{ij}^k)^2 \right)
		\end{aligned}
	\end{equation}
	for any $i,j = 1,\ldots,n$, with $i\neq j$.
\end{remark}

\section{Lagrangian submanifolds of $Q^n$ and hypersurfaces of $S^{n+1}(1)$}

Let $a: M^n \to S^{n+1}(1) \subset \R^{n+2}$ be an immersion and denote by $b$ a unit normal vector field along this immersion, tangent to $S^{n+1}(1)$. Let $\lambda_1,\ldots,\lambda_n$ be the principal curvatures and denote by $\{e_1,\ldots,e_n\}$ a local orthonormal frame given by principal directions such that $Se_j = \lambda_j e_j$ for $j=1,\ldots,n$, where $S$ is the shape operator associated with $b$. The Gauss map of the hypersurface $a$ is given by
\begin{equation} \label{def_G}
	G: M^n \to Q^n: p \mapsto [a(p)+\sqrt{-1} b(p)].
\end{equation}
Remark that
\begin{equation} \label{def_Ghat}
	\hat G: M^n \to V_2(\R^{n+2}): p \mapsto \frac{1}{\sqrt 2}(a(p)+\sqrt{-1}b(p))
\end{equation}
is a map into the Stiefel manifold $V_2(\R^{n+2})$ such that $G = \pi \circ \hat G$, which shows that $G$ indeed takes values in $Q^n$. In fact, $\hat G$ is horizontal since
\begin{equation} \label{dGhat}
	(d\hat G)e_j = \frac{1}{\sqrt 2}(1-\sqrt{-1}\lambda_j) e_j
\end{equation}
is perpendicular to $\sqrt{-1}\hat G$ for all $j=1,\ldots,n$. It also follows from \eqref{dGhat} that $ (d\hat G)e_j $ is perpendicular to $\sqrt{-1}(d\hat G)e_k $ for all $j,k=1,\ldots,n$, which implies that $G$ is Lagrangian. The map $\hat G$ is, up to multiplication with a factor $e^{\sqrt{-1}t}$ for some constant $t\in \mathbb R$, the unique horizontal lift of $G$.

If the hypersurface $a$ is \textit{isoparametric} in $S^{n+1}(1)$, i.e., if the principal curvatures $\lambda_1,\ldots,\lambda_n$ are constant, then the Gauss map is a minimal Lagrangian immersion. This follows either by a straightforward computation of the second order derivatives of $\hat G$ or from the following elegant formula from \cite{Palmer}:
\begin{equation}
	g(J\vec{H},\cdot) = -\frac{1}{n} \ d \left( \mathrm{Im} \left( \log \prod_{j=1}^n (1+\sqrt{-1}\lambda_j) \right) \right),
\end{equation}
where $\vec{H}$ is the mean curvature vector of the Gauss map.

As remarked in \cite{MaOhnita1},  any Lagrangian immersion $f:M^n \to Q^n$ can locally be seen as the Gauss map of a hypersurface of $S^{n+1}(1)$. Indeed, inspired by \eqref{def_Ghat}, we can always take a horizontal lift $\hat f: M^n \to V_2(\R^{n+2})$ such that ($\sqrt 2$ times) its real part is locally an immersion into $S^{n+1}(1)$. In the following, we prove Theorem \ref{theo1}.

\medskip
\noindent \textbf{Proof of Theorem~\ref{theo1}:}
As discussed above, it is well-known that the Gauss map $G: M^n \to Q^n: p \mapsto [a(p)+\sqrt{-1}b(p)]$ of an isoparametric hypersurface $a: M^n \to S^{n+1}(1)$ with unit normal $b$ is a minimal Lagrangian immersion. Let $A_0 \in \mathcal A$ be as in Example \ref{ex1} for the lift $\hat G$ given in \eqref{def_Ghat} and let $A \in \mathcal A$ be arbitrary. Lemma \ref{lem4} implies that, along the image of $G$, we have $A = \cos\varphi \, A_{0} + \sin\varphi \, JA_{0}$ for some function $\varphi: M^n \to \R$. Let $\{e_1,\ldots,e_n\}$ be a local orthonormal frame of principal directions for the immersion $a$ on $M^n$. Then, using \eqref{A0} and \eqref{dGhat},
\begin{align*}
	A_0 &(dG)e_j = -(d\pi) \left( \overline{(d \hat G)e_j} \right) = -(d\pi) \left( \frac{1}{\sqrt 2}(1+\sqrt{-1}\lambda_j)e_j \right) \\
	& = -(d\pi) \left( \frac{1-\lambda_j^2}{1+\lambda_j^2} \, (d \hat G)e_j + \frac{2\lambda_j}{1+\lambda_j^2} \, \sqrt{-1} (d \hat G)e_j \right) = \frac{\lambda_j^2-1}{\lambda_j^2+1} \, (dG)e_j - \frac{2\lambda_j}{\lambda_j^2+1} \, J (dG)e_j.
\end{align*}
This implies that the frame $\{e_1, \ldots, e_n \}$ diagonalizes the operators $B_0$ and $C_0$, associated with $A_0$ as explained above, and that the angle functions are determined by
\begin{equation} \label{theta0}
	\cos(2\theta^0_j) = \frac{\lambda_j^2-1}{\lambda_j^2+1}, \qquad
	\sin(2\theta^0_j) = \frac{2\lambda_j}{\lambda_j^2+1}.
\end{equation}
\eqref{theta0} implies that $\lambda_j=\cot\theta^0_j$. In particular, we obtain that $\theta^0_1,\ldots,\theta^0_n$ are constant. It follows from Lemma \ref{lem4} that the angle functions associated with $A$ are given by $\theta_j = \theta_j^0 - \varphi/2$, which implies that the difference between any two of them is constant. Now assume that $A$ is chosen as in Example \ref{ex2}. Since $\theta_1 + \cdots + \theta_n = 0 \mod \pi$, we obtain that $\varphi$ is also constant and $\lambda_j = \cot(\theta_j^0) = \cot(\theta_j + \varphi/2)$ for $j=1,\ldots,n$.

Conversely, let $f: M^n \to Q^n$ be a minimal Lagrangian immersion with constant angle functions associated with some $A \in \mathcal A$. Since $M^n$ is simply connected, we can take a horizontal lift $\hat f: M^n \to V_2(\R^{n+2})$ of $f$, which can be written as $\hat f = (a+\sqrt{-1}b)/\sqrt 2$ and hence defines two maps $a,b: M^n \to S^{n+1}(1)$ such that $a(p)$ and $b(p)$ are orthogonal for every $p \in M^n$. For every constant $t \in \R$, there is another horizontal lift of $f$, namely,
\begin{equation} \label{liftf}
	\hat f_t = \frac{e^{\sqrt{-1}t}}{\sqrt 2}(a+\sqrt{-1}b) = \frac{1}{\sqrt 2}(a_t+\sqrt{-1}b_t),
\end{equation}
where $a_t= (\cos t) a - (\sin t) b$ and $b_t= (\sin t) a + (\cos t) b$. Now let $\xi$ be a unit normal vector field to $Q^n$ such that $A$ is the shape operator associated with $\xi$. We can lift the restriction of $\xi$ to the image of $f$ to a horizontal vector field along the image of $\hat f$, which can be written as $\hat\xi = e^{\sqrt{-1}\varphi}(a-\sqrt{-1}b)/\sqrt 2$ for some function $\varphi: M^n \to \R$. We know from Corollary \ref{cor1} that, for every vector $X$ tangent to $M^n$, one has $\nabla^{\perp}_X \xi = 0$, where $\nabla^{\perp}$ is the normal connection of $Q^n$ in $\mathbb{C} P^{n+1}(4)$. This implies that also $\nabla^{\perp}_X \hat\xi = 0$, where $\nabla^{\perp}$ is now the normal connection of $V_2(\R^{n+2})$ in $S^{2n+3}(1)$. Combining this with the expression for $\hat\xi$ yields that $\varphi$ is constant. If we lift $\xi$ to the immersion $\hat f_t$ rather than to $\hat f$, we obtain
\begin{equation} \label{liftxi}
	\hat{\xi}_t = e^{\sqrt{-1}t}\hat{\xi} = \frac{e^{\sqrt{-1}(\varphi+t)}}{\sqrt 2}(a-\sqrt{-1}b) = \frac{e^{\sqrt{-1}(\varphi+2t)}}{\sqrt 2}(a_t-\sqrt{-1}b_t).
\end{equation}
It follows from \eqref{liftf} and \eqref{liftxi} that
\begin{equation} \label{eq_ab}
	a_t = \frac{1}{\sqrt 2} (\hat f_t + e^{-\sqrt{-1}(\varphi+2t)} \hat{\xi}_t), \qquad b_t = -\frac{\sqrt{-1}}{\sqrt 2}(\hat f_t - e^{-\sqrt{-1}(\varphi+2t)} \hat{\xi}_t).
\end{equation}
Let us investigate when $a_t$ is an immersion. Let $\{e_1,\ldots,e_n\}$ be an orthonormal frame such that $A(df)e_j = \cos(2\theta_j) (df)e_j - \sin(2\theta_j) J(df)e_j$. Since $(d\pi)(d\hat{\xi})e_j=-A(df)e_j$, we obtain
\begin{equation} \label{eq_da}
	\begin{aligned}
		(da_t)e_j &= \frac{1}{\sqrt 2} \left( (d\hat f_t)e_j - e^{-\sqrt{-1}(\varphi+2t)} \left( \cos(2\theta_j) (d\hat f_t)e_j - \sqrt{-1}\sin(2\theta_j) (d \hat f_t)e_j \right) \right) \\
		&= \frac{1}{\sqrt 2} \left( 1-e^{-\sqrt{-1}(2\theta_j+\varphi+2t)} \right)(d\hat f_t)e_j,
	\end{aligned}
\end{equation}
which means that $a_t$ is an immersion if and only if $2\theta_j + \varphi + 2t$ is not a multiple of $2\pi$ for any $j=1,\ldots,n$. Consequently, the equation for $a_t$ in \eqref{eq_ab} defines an immersion into $S^{n+1}(1)$ for every choice of constant $t$, and hence of $c=\varphi/2+t$, satisfying $\sin(\theta_j + \varphi/2 + t) \neq 0$ for any $j=1,\ldots,n$. Moreover, $b_t$, as defined by \eqref{eq_ab}, must be a unit normal to this hypersurface tangent to the sphere, since it is perpendicular to $a_t$ and also to $(da_t)e_j$ for all $j=1,\ldots,n$, as can be seen from \eqref{eq_da}. This means that $a_t$ is a hypersurface with Gauss map $f$. By comparing \eqref{dGhat} and \eqref{eq_da}, we see that the principal curvatures of this hypersurface are given by $\lambda_j = \cot(\theta_j + \varphi/2 + t)$. In particular, they are constant and hence $a_t$ is isoparametric.
\qed

\begin{remark}
	It follows from Theorem \ref{theo1} that for a given Lagrangian immersion $f: M^n \to Q^n$, which is minimal and, after choosing $A$ as in Example \ref{ex2}, has constant angle functions, there are several isoparametric hypersurfaces of $S^{n+1}(1)$ with Gauss map $f$. It follows from the proof that if $a$ is such a hypersurface, with unit normal $b$, then $a_t = (\cos t) a + (\sin t) b$ will define a hypersurface with the same Gauss map for any $t \in \mathbb R$, provided that $a_t$ is an immersion. Indeed, if $a$ is $\sqrt 2$ times the real part of a horizontal lift $\hat f$, then $a_t$ is $\sqrt 2$ times the real part of the horizontal lift $\hat f_t = e^{\sqrt{-1}t} \hat f$. Remark that $a_t$ is a so-called \textit{parallel hypersurface} of $a$ in $S^{n+1}(1)$.
\end{remark}

\begin{remark} The formula $\lambda_j = \cot\theta_j$ appeared before in the theory of isoparametric hypersurfaces, for example in M\"unzner's paper \cite{MunznerI}. Theorem \ref{theo1} gives an interpretation for the angles $\theta_j$.
\end{remark}

We can now translate everything that is known about the classification of isoparametric hypersurfaces of spheres to the theory of minimal Lagrangian submanifolds of $Q^n$. For example, we have the following.

\begin{corollary} \label{cor2}
	Let $f: M^n \to Q^n$ be a minimal Lagrangian immersion with constant angle functions. If $g$ is the number of different constant angle functions modulo $\pi$, then $g \in \{1,2,3,4,6\}$. Moreover,
	\begin{itemize}
		\item if $g=1$, then $f$ is the Gauss map of a part of the standard embedding $S^n(r) \to S^{n+1}(1)$;
		\item if $g=2$, then $f$ is the Gauss map of a part of the standard embedding  $S^k(r_1) \times S^{n-k}(r_2) \to S^{n+1}(1)$;
		\item if $g=3$, then $f$ is the Gauss map of a part of a tube around the standard embedding $\mathbb R P^2 \to S^4(1)$, $\mathbb C P^2 \to S^7(1)$, $\mathbb H P^2 \to S^{13}(1)$ or $\mathbb O P^2 \to S^{25}(1)$ (in the first case, the standard embedding is a Veronese embedding), which are known as Cartan's isoparametric hypersurfaces.
	\end{itemize}
\end{corollary}

\begin{proof}
	The first part of the statement follows from M\"unzner's theorem \cite{MunznerI, MunznerII} on the possible numbers of distinct principal curvatures of an isoparametric hypersurface of $S^{n+1}(1)$. Indeed, we know from Theorem \ref{theo1} that $f$ is the Gauss map of an isoparametric hypersurface of $S^{n+1}(1)$ and that the number of different constant angle functions modulo $\pi$ equals the number of distinct principal curvatures of this hypersurface. The second part follows from the classification of isoparametric hypersurfaces of spheres for $g=1$, $g=2$ and $g=3$, known since the work of Cartan \cite{C}.
\end{proof}

\begin{remark}
	The corollary above is only a partial result in the sense that we can translate everything which is known about the classification of isoparametric hypersurfaces of spheres to minimal Lagrangian submanifolds of $Q^n$, also for the cases $g=4$ and $g=6$.
\end{remark}

The following theorem states that the first two  examples in Corollary \ref{cor2} are the only totally geodesic Lagrangian submanifolds of $Q^n$.

\begin{theorem} \label{theoTG}
	Let $f: M^n \to Q^n$ be a totally geodesic Lagrangian immersion. Then $f$ is the Gauss map of a part of the standard embedding $S^n(r) \to S^{n+1}(1)$ or $S^k(r_1) \times S^{n-k}(r_2) \to S^{n+1}(1)$. In the former case, the metric induced by $f$ gives $M^n$ constant sectional curvature $2$. In the latter case, the metric induced by $f$ does not give $M^n$ constant sectional curvature, unless $k=1$ and $n=2$, in which case $f$ makes $M^2$ flat.
\end{theorem}

\begin{proof}
	Choose $A \in \mathcal A$ as in Example \ref{ex2}. Then it follows from Corollary \ref{cor1} that the one-form $s$ vanishes on tangent bundle of $M^n$ and hence from \eqref{intcond1} that all the angle functions are constant. We then know from Theorem \ref{theo1} that $f$ is the Gauss map of an isoparametric hypersurface of $S^{n+1}(1)$ and from Corollary \ref{cor2} that the number of different constant angle functions modulo $\pi$ is $1$, $2$, $3$, $4$ or $6$.
	
	The equation of Codazzi \eqref{codazzi} for $X=e_i$, $Y=e_j$ en $Z=e_k$ yields $$\sin(2(\theta_i-\theta_j))(\delta_{jk}Je_i + \delta_{ik}Je_j) = 0,$$
	which implies that $2\theta_i = 2\theta_j \mod \pi$ for all indices $i$ and $j$. This means that there can be at most two different angle functions modulo $\pi$ and we obtain the result from Corollary \ref{cor2}.
	
	The claims about the sectional curvature follow from \eqref{seccurv}.
\end{proof}

The following theorem states that the first example in Corollary \ref{cor2} describes the only family of minimal Lagrangian submanifolds of $Q^n$ for which all angle functions are equal.

\begin{theorem} \label{theoSameAngles}
	Let $f: M^n \to Q^n$ be a minimal Lagrangian immersion such that all angle functions are equal. Then $f$ is the Gauss map of a part of the standard embedding $S^n(r) \to S^{n+1}(1)$.
\end{theorem}

\begin{proof}
	Assume that the angle functions appearing in the statement of the theorem correspond to $A_0 \in \mathcal A$ and choose the almost product structure $A \in \mathcal A$ as in Example \ref{ex2}. By Lemma \ref{lem4}, the angle functions corresponding to $A$ are still all equal and, since their sum vanishes modulo $\pi$, they are all constant. By Theorem \ref{theo1}, the immersion is the Gauss map of a totally umbilical hypersurface of $S^{n+1}(1)$, which proves the theorem.
\end{proof}

The following theorem explicitly describes minimal Lagrangian  immersions with constant angle functions in $Q^3$.

\begin{theorem} \label{theoCAdim3}
	Let $f:M^3 \to Q^3$ be a minimal Lagrangian immersion and choose $A \in \mathcal A$ as in Example \ref{ex2}. Assume that the corresponding angle functions are constant and denote by $g$ the number of different constant angle functions modulo $\pi$. Then
	$f$ is one of the following:
	\begin{itemize}
		\item[(i)] $f$ is the Gauss map of a part of the standard embedding $S^3(r) \to S^4(1)$ if $g=1$;
		\item[(ii)] $f$ is the Gauss map of a part of the standard embedding $S^1(r_1) \times S^2(r_2) \to S^4(1)$ if $g=2$;
		\item[(iii)] $f$ is the Gauss map of a part of one of Cartan's isoparametric hypersurfaces: a tube around the Veronese surface in $S^4(1)$ if $g=3$.
	\end{itemize}
	In the third case, the metric induced by $f$ gives $M^3$ constant sectional curvature $1/8$.
\end{theorem}

\begin{proof}
	Corollary \ref{cor1} and equation \eqref{intcond1} imply that all components of the second fundamental form for which at least two indices are the same, vanish. The only possibly non-zero component of $h$ is thus $h_{12}^3$ and we distinguish two cases.
	
	\emph{Case 1: $h_{12}^3 = 0$.} In this case, $f$ is totally geodesic and hence it is the Gauss map of the standard embedding $S^3(r) \to S^4(1)$ or $S^1(r_1) \times S^2(r_2) \to S^4(1)$ by Theorem \ref{theoTG}.
	
	\emph{Case 2: $h_{12}^3 \neq 0$.} In this case, it follows from equation \eqref{intcond2} that the constant angle functions $\theta_1$, $\theta_2$ and $\theta_3$ are mutually different. It then follows from Corollary \ref{cor2} that $f$ is the Gauss map of one of Cartan's isoparametric hypersurfaces of $S^4(1)$. Considering the $Je_1$-component of the Codazzi equation \eqref{codazzi} for $X=e_1$ and $Y=Z=e_2$, using \eqref{intcond2} and some elementary trigonometric identities, we obtain
	\begin{equation} \label{(h123)_1}
		(h_{12}^3)^2 = -\cos(\theta_1-\theta_2) \sin(\theta_2-\theta_3) \sin(\theta_3-\theta_1).
	\end{equation}
	Similarly, the $Je_2$-component of the Codazzi equation for $X=e_2$ and $Y=Z=e_3$ and the $Je_3$-component of the Codazzi equation for $X=e_3$ and $Y=Z=e_1$ yield
	\begin{align}
		& (h_{12}^3)^2 = -\sin(\theta_1-\theta_2) \cos(\theta_2-\theta_3) \sin(\theta_3-\theta_1), \label{(h123)_2} \\
		& (h_{12}^3)^2 = -\sin(\theta_1-\theta_2) \sin(\theta_2-\theta_3) \cos(\theta_3-\theta_1). \label{(h123)_3}
	\end{align}
	Combining equations \eqref{(h123)_1}--\eqref{(h123)_3}, yields $\{\theta_1,\theta_2,\theta_3\} = \{0,\pi/3,-\pi/3\}$, as always modulo $\pi$, and hence $(h_{12}^3)^2 = 3/8$. Finally, from \eqref{seccurv}, we obtain that the sectional curvature of any plane $\mathrm{span}\{e_i,e_j\}$ is given by $K_{ij} = 2\cos^2(\theta_i-\theta_j)-(h_{12}^3)^2 = 1/8$.
\end{proof}

\section{Minimal Lagrangian submanifolds of $Q^n$ with constant sectional curvature}

The classifications in Theorem \ref{theoTG} and Theorem \ref{theoCAdim3} include examples of minimal Lagrangian submanifolds of $Q^n$ with constant sectional curvature: the Gauss map of a round sphere and the Gauss map of one of Cartan's examples. In this section, we prove Theorem \ref{theoCSCconclusion}, i.e., we classify all minimal Lagrangian submanifolds of $Q^n$ with constant sectional curvature, for arbitrary $n \geq 2$.  This classification can be regarded as a counterpart of the classic result by Ejiri \cite{Ejiri1982} for the case of complex space form, while the proof is completely different from that in \cite{Ejiri1982}.

\subsection{General results}

We first prove some results which are valid for any dimension $n \geq 2$. This is the key step of the proof of Theorem \ref{theoCSCconclusion}.

\begin{lemma}\label{lem5.1}
	
	Let $f: M^n \to Q^n$ be a  Lagrangian submanifold with constant sectional curvature. Assume that an almost product structure $A=\cos\varphi A_{\eta}+\sin\varphi JA_{\eta} \in \mathcal A$ is fixed on $Q^n$ and let $\{e_1,\ldots,e_n\}$ be a local orthonormal frame on $M^n$ diagonalizing  the associated operators $B$ and $C$. Denote by $\theta_1,\ldots,\theta_n$ the angle functions as defined above. Then
	\begin{equation} \label{CSCeq0}
		\begin{aligned}
			& \sin(\theta_i-\theta_j) \sin(\theta_i+\theta_j-2\theta_k) (\delta_{k \ell}h(e_i,e_j) + h_{ij}^{\ell} Je_k) \\
			& + \sin(\theta_j-\theta_k) \sin(\theta_j+\theta_k-2\theta_i) (\delta_{i \ell}h(e_j,e_k) + h_{jk}^{\ell} Je_i) \\
			& + \sin(\theta_k-\theta_i) \sin(\theta_k+\theta_i-2\theta_j) (\delta_{j \ell}h(e_i,e_k) + h_{ik}^{\ell} Je_j) = 0
		\end{aligned}
	\end{equation}
	for all $i,j,k,\ell = 1,\ldots,n$. In particular,
	\begin{align}
		\label{CSCeq1} & h_{ii}^k \sin(\theta_i-\theta_k) \sin(\theta_i+\theta_k-2\theta_j) = h_{jj}^k \sin(\theta_j-\theta_k) \sin(\theta_j+\theta_k-2\theta_i), \\
		\label{CSCeq2} & h_{ij}^k \sin(\theta_i-\theta_j) \sin(\theta_i+\theta_j-2\theta_k) = 0
	\end{align}
	for $i,j,k = 1,\ldots,n$ mutually different, and
	\begin{align}
		\label{CSCeq3} & h_{ij}^k \sin(\theta_i-\theta_j) \sin(\theta_i+\theta_j-2\theta_{\ell}) = 0
	\end{align}
	for $i,j,k,\ell = 1,\ldots,n$ mutually different.
\end{lemma}

\begin{proof} We start by taking the covariant derivative of the Codazzi equation \eqref{codazzi} along a vector field $W$:
	\begin{equation} \label{codazzider}
		(\overline\nabla^2 h)(W,X,Y,Z) - (\overline\nabla^2 h)(W,Y,X,Z) = (\overline\nabla T)(W,X,Y,Z),
	\end{equation}
	where  $T$ is the $(1,3)$-tensor field taking values in the normal bundle, given by
	\begin{equation} \label{eqT}
		T(X,Y,Z) = g(CY,Z) JBX - g(CX,Z) JBY - g(BY,Z) JCX + g(BX,Z) JCY
	\end{equation}
	for all $X$, $Y$ and $Z$ tangent to $M^n$.
	If we take a cyclic sum of \eqref{codazzider} over $W$, $X$ and $Y$, keeping $Z$ fixed, we see that the left hand side vanishes. Indeed, using the Ricci identity, denoting by $R^{\perp}$ the curvature tensor of the normal connection of $M^n$ in $Q^n$, by $R$ the curvature tensor of $M^n$ and by $c$ the constant sectional curvature of $M^n$, we have
	\begin{eqnarray*}
		\lefteqn{\sum_{W,X,Y}^{\mathrm{cyclic}} \left( (\overline\nabla^2 h)(W,X,Y,Z) - (\overline\nabla^2 h)(W,Y,X,Z) \right) } \\
		&=& \sum_{W,X,Y}^{\mathrm{cyclic}} \left( (\overline\nabla^2 h)(W,X,Y,Z) - (\overline\nabla^2 h)(X,W,Y,Z) \right) \\
		&=& \sum_{W,X,Y}^{\mathrm{cyclic}} \left( R^\perp(W,X)h(Y,Z) - h(R(W,X)Y,Z) - h(Y,R(W,X)Z) \right) \\
		&=& \sum_{W,X,Y}^{\mathrm{cyclic}} \left( -JR(W,X)Jh(Y,Z) - h(R(W,X)Y,Z) - h(Y,R(W,X)Z) \right) \\
		&=& -c \sum_{W,X,Y}^{\mathrm{cyclic}} \left( g(X,Jh(Y,Z))JW - g(W,Jh(Y,Z))JX + g(X,Y)h(W,Z) \right. \\
		& & \hspace{1.5cm} \left. - g(W,Y)h(X,Z) + g(X,Z)h(Y,W) - g(W,Z) h(Y,X) \right) \\
		&=& 0.
	\end{eqnarray*}
	This implies that
	\begin{equation} \label{eqnablaT0}
		\sum_{W,X,Y}^{\mathrm{cyclic}} (\overline\nabla T)(W,X,Y,Z) = 0.
	\end{equation}
	From \eqref{eqT}, we obtain immediately that
	\begin{eqnarray*}
		(\overline\nabla T)(W,X,Y,Z) &=& g((\nabla_WC)Y,Z)JBX + g(CY,Z)J(\nabla_WB)X \\
		&& - g((\nabla_WC)X,Z) JBY - g(CX,Z) J(\nabla_WB)Y \\
		&& - g((\nabla_WB)Y,Z) JCX - g(BY,Z) J(\nabla_WC)X \\
		&& + g((\nabla_WB)X,Z) JCY + g(BX,Z) J(\nabla_WC)Y,
	\end{eqnarray*}
	which, by \eqref{nablab} and \eqref{nablac}, is equivalent to	
	\begin{eqnarray*}
		(\overline\nabla T)(W,X,Y,Z) &=& -g(Jh(W,BY),Z)JBX - g(BJh(W,Y),Z)JBX \\
		&& - g(CY,Z)h(W,CX) + g(CY,Z)JCJh(W,X) \\
		&& + g(Jh(W,BX),Z)JBY + g(BJh(W,X),Z)JBY \\
		&& + g(CX,Z)h(W,CY) - g(CX,Z)JCJh(W,Y) \\
		&& - g(Jh(W,CY),Z)JCX - g(CJh(W,Y),Z)JCX \\
		&& - g(BY,Z)h(W,BX) + g(BY,Z)JBJh(W,X) \\
		&& + g(Jh(W,CX),Z)JCY + g(CJh(W,X),Z)JCY \\
		&& + g(BX,Z)h(W,BY) - g(BX,Z)JBJh(W,Y).
	\end{eqnarray*}	
	Remark that the terms involving $s(W)$ cancel two by two. When taking the cyclic sum over $W$, $X$ and $Y$, the expression simplifies to
	\begin{equation} \label{eqnablaT1}
		\begin{aligned}
			\sum_{W,X,Y}^{\mathrm{cyclic}} (\overline\nabla T)(W,X,Y,Z) = \sum_{W,X,Y}^{\mathrm{cyclic}} ( & -g(Jh(W,BY),Z)JBX - g(CY,Z)h(W,CX) \\
			& + g(Jh(W,BX),Z)JBY + g(CX,Z)h(W,CY) \\
			& - g(Jh(W,CY),Z)JCX - g(BY,Z)h(W,BX) \\
			& + g(Jh(W,CX),Z)JCY + g(BX,Z)h(W,BY) ). \\
		\end{aligned}
	\end{equation}
	By combining \eqref{eqnablaT0} and \eqref{eqnablaT1} for $W=e_i$, $X=e_j$, $Y=e_k$ and $Z=e_{\ell}$, we obtain \eqref{CSCeq0}.
	
	By taking $i$, $j$ and $k$ mutually different and $\ell=k$ in \eqref{CSCeq0}, we obtain \eqref{CSCeq1}, \eqref{CSCeq2} and \eqref{CSCeq3} up to renaming the indices.
\end{proof}

\begin{proposition} \label{propCSC1}
	Let $f:M^n \to Q^n$ be a minimal Lagrangian immersion such that $M^n$ has constant sectional curvature and choose $A \in \mathcal A$ as in Example \ref{ex2}. Then the local angle functions are either all the same or mutually different modulo $\pi$. In the former case, the immersion is totally geodesic and is the Gauss map of a part of the standard embedding $S^n(r) \to S^{n+1}(1)$.
\end{proposition}

\begin{proof}
	Assume first that all angle functions are the same modulo $\pi$. Then the last part of the proposition follows from Theorem \ref{theoSameAngles}.
	
	Now consider the case that at least two angle functions are different modulo $\pi$. We have to prove that this implies that \emph{all} the angle functions are mutually different modulo $\pi$. This is trivial for $n=2$, so we assume from now on that $n \geq 3$. Proceeding by contradiction, we assume that $\theta_1 = \cdots = \theta_m \mod \pi$ for some $m \in \{2,\ldots,n-1\}$ and that $\theta_{\ell} \neq \theta_1 \mod \pi$ for all $\ell > m$.
	
	\textit{Step 1: If $X,Y \in \mathrm{span}\{e_1,\ldots,e_m\}$ and $X \perp Y$, then $h(X,Y)=0$.} After changing the orthonormal frame $\{e_1,\ldots,e_m\}$ if necessary, we may assume that $X$ is a scalar multiple of $e_1$ and $Y$ is a scalar multiple of $e_2$. It suffices to prove that $h_{12}^{\ell}=0$ for any $\ell \in \{1,\ldots,n\}$. If $\ell \leq m$, putting $i=\ell$, $j=1$ and $k=2$ in \eqref{intcond2}, gives $h_{12}^{\ell}=0$. If $\ell > m$ on the other hand, taking $i=1$, $j=\ell$ and $k=2$ in \eqref{CSCeq2}, gives $h_{12}^{\ell}=0$. This proves the claim.
	
	\textit{Step 2: If $X,Y \in \mathrm{span}\{e_1,\ldots,e_m\}$ and $\|X\|=\|Y\|$, then $h(X,X)=h(Y,Y)$.} This follows immediately from Step 1 by noting that $X+Y \perp X-Y$ and using the bilinearity and the symmetry of $h$.
	
	\textit{Step 3: $M^n$ cannot have constant sectional curvature.} By using \eqref{seccurv}, Step 1 and Step 2, we obtain that the sectional curvature of the plane spanned by $e_1$ and $e_2$ satisfies
	$$ K_{12} = 2 + g(h(e_1,e_1),h(e_2,e_2)) - g(h(e_1,e_2),h(e_1,e_2)) = 2 + \|h(e_1,e_1)\|^2 \geq 2. $$
	On the other hand, if $\ell > m$, then, again by using \eqref{seccurv}, the sectional curvature of the plane spanned by $e_1$ and $e_{\ell}$ satisfies
	$$ K_{1 \ell} = 2 \cos^2(\theta_1-\theta_{\ell}) + g(h(e_1,e_1),h(e_{\ell},e_{\ell})) - g(h(e_1,e_{\ell}),h(e_1,e_{\ell})) < 2 + g(h(e_1,e_1),h(e_{\ell},e_{\ell})). $$
	Remark that the inequality is strict since $\theta_1 \neq \theta_{\ell} \mod \pi$. It is now sufficient to prove that there is at least one $\ell_0 > m$ for which $g(h(e_1,e_1),h(e_{\ell_0},e_{\ell_0})) \leq 0$ since, for this particular index $\ell_0$, one has $K_{1 \ell_0} < 2$ and hence $K_{1 \ell_0} \neq K_{12}$. To prove the existence of such an index $\ell_0$, we remark that, due to minimality and Step 2,
	$$ m \, h(e_1,e_1) + \sum_{\ell = m+1}^n h(e_{\ell},e_{\ell}) = 0. $$
	Taking the inner product with $h(e_1,e_1)$ yields
	$$ m \, \|h(e_1,e_1)\|^2 +  \sum_{\ell = m+1}^n g(h(e_1,e_1),h(e_{\ell},e_{\ell})) = 0,$$
	which shows that indeed one of the terms in the sum must be non-positive.
	
	This contradiction completes the proof.
\end{proof}

\begin{proposition} \label{propCSC2}
	Let $f:M^n \to Q^n$ be a minimal Lagrangian immersion such that $M^n$ has constant sectional curvature. If there exist three mutually different indices $i$, $j$ and $k$ such that $h_{ij}^k \neq 0$, then $n=3$ and the immersion is part of the Gauss map of a tube around the Veronese surface in $S^4(1)$.
\end{proposition}

\begin{proof}
	Without loss of generality, we assume that $h_{12}^3 \neq 0$. We will prove that the assumption $n \geq 4$ leads to a contradiction. Remark that, since $f$ is not totally geodesic, it follows from Proposition \ref{propCSC1} that all the angle functions are different modulo $\pi$. Applying  \eqref{CSCeq2} for $(i,j,k)=(1,2,3)$, we obtain $\sin(\theta_1+\theta_2-2\theta_3)=0$ and applying  \eqref{CSCeq3} for $(i,j,k,\ell)=(1,2,3,4)$, we obtain $\sin(\theta_1+\theta_2-2\theta_4)=0$. By combining these two equations, we find that $2\theta_3=2\theta_4 \mod \pi$. Again from \eqref{CSCeq3}, but now for $(i,j,k,\ell)=(1,3,2,4)$, we obtain $\sin(\theta_1+\theta_3-2\theta_4)=0$. By combining this with $2\theta_3=2\theta_4 \mod \pi$, we obtain $\sin(\theta_1-\theta_3)=0$, and hence $\theta_1=\theta_3\mod\pi$, which is a contradiction.
	Hence, we obtain that $n=3$.
	
	To prove the second part of the theorem, it suffices, in view of Theorem \ref{theoCAdim3} and its proof, to prove that all the angle functions are constant. We choose $A \in \mathcal A$ as in Example \ref{ex2}, such that $\theta_1+\theta_2+\theta_3=0 \mod\pi$. By taking $(i,j,k)=(1,2,3)$, respectively $(i,j,k)=(1,3,2)$, in \eqref{CSCeq2} and using $h_{12}^3 \neq 0$, we obtain $\theta_1+\theta_2-2\theta_3=0 \mod\pi$, respectively $\theta_1+\theta_3-2\theta_2=0 \mod\pi$. Combining all three relations for $\theta_1$, $\theta_2$ and $\theta_3$ implies that all three functions are constant.
\end{proof}

\subsection{Classification in dimension $n=2$}

The complex hyperquadric $Q^2$ is isometric to the product of spheres $S^2(1/2) \times S^2(1/2)$, as can be deduced for example from \cite{CU} or \cite{TU2015}. Obviously, the almost product structure related to this splitting is different from the non-integrable almost product structures in $\mathcal A$. It was proven in \cite{CU} that a minimal Lagrangian surface with constant Gaussian curvature in $S^2(1) \times S^2(1)$ must be totally geodesic. In combination with Theorem \ref{theoTG}, we obtain the following.

\begin{proposition} \label{propCSCn=2}
	Let $f:M^2 \to Q^2$ be a minimal Lagrangian immersion such that $M^2$ has constant Gaussian curvature. Then the immersion is totally geodesic and $f$ is the Gauss map of the standard embedding $S^2(r) \to S^3(1)$ or $S^1(r_1) \times S^1(r_2) \to S^3(1)$. In the first case, $M^2$ has constant Gaussian curvature $2$ and in the second case, $M^2$ is flat.
\end{proposition}

\subsection{Classification in dimension $n=3$}

The following theorem gives a complete classification of minimal Lagrangian submanifolds of $Q^3$ with constant sectional curvature.

\begin{proposition}\label{propCSCn=3}
	Let $f:M^3\to Q^3$ be a Lagrangian minimal immersion with constant sectional curvature, then either $M^3$ has constant sectional curvature $2$ and $f$ is the Gauss map of a part of the standard embedding $S^3(r) \to S^4(1)$, or $M^3$ has constant sectional curvature $1/8$ and $f$ is the Gauss map of a part of a tube around a Veronese surface in $S^4(1)$.
\end{proposition}

\begin{proof}
	We choose $A \in \mathcal A$ as in Example \ref{ex2}, then we have that $\theta_1+\theta_2+\theta_3=0 \mod \pi$. It follows from Proposition \ref{propCSC1} that the angle functions are either all equal modulo $\pi$ or mutually different  modulo $\pi$ and that, in the former case, we obtain the first case in the proposition. Assume from now on that all the angle functions $\theta_1, \theta_2, \theta_3$ are different modulo $\pi$.
	
	If $h_{12}^3 \neq 0$, it follows from Proposition \ref{propCSC2} that we obtain the Gauss map of a part of a tube around a Veronese surface in $S^4(1)$. We will assume from now on that $h_{12}^3=0$ and prove that no new examples can occur.  We denote
	\begin{equation*}
		\begin{aligned}
			x&=\sin{(\theta_1-\theta_2)}\sin{(\theta_1+\theta_2-2\theta_3)},\\
			y&=\sin{(\theta_2-\theta_3)}\sin{(\theta_2+\theta_3-2\theta_1)},\\
			z&=\sin{(\theta_3-\theta_1)}\sin{(\theta_3+\theta_1-2\theta_2)}.
		\end{aligned}
	\end{equation*}
	By using elementary trigonometric identities, we have that $x+y+z = 0$. Moreover, \eqref{CSCeq1} is equivalent to
	\begin{equation} \label{dim3eq1}
		h_{22}^1 x + h_{33}^1 z = 0, \qquad h_{11}^2 x + h_{33}^2 y = 0, \qquad h_{22}^3 y + h_{11}^3 z = 0.
	\end{equation}
	We now distinguish three cases.
	
	\textit{Case 1: At least two of the local functions $x$, $y$ and $z$ are identically zero.} Since the sum of the three functions vanishes, we know that all three of them are identically zero. Moreover, since the local angle functions $\theta_1$, $\theta_2$ and $\theta_3$ are mutually different modulo $\pi$, it follows that $\theta_1+\theta_2-2\theta_3$, $\theta_2+\theta_3-2\theta_1$ and $\theta_3+\theta_1-2\theta_2$ are all integer multiples of $\pi$. Together with $\theta_1+\theta_2+\theta_3=0\mod\pi$, this implies that the angle functions are all constants that are different modulo $\pi$ and it follows from Theorem \ref{theoCAdim3} that $f$ is the Gauss map of a part of a tube around a Veronese surface in $S^4(1)$, which contradicts $h_{12}^3=0$.
	
	\textit{Case 2: Exactly one of the local functions $x$, $y$ and $z$ is identically zero.} Without loss of generality, we may assume that $x=0$, Remark that $y=-z \neq 0$. It follows from \eqref{dim3eq1} that $h_{33}^1=h_{33}^2=0$ and $h_{11}^3=h_{22}^3$. Since $x=0$, we have $\theta_1+\theta_2-2\theta_3=0 \mod\pi$ and deriving this equality, using \eqref{intcond1}, yields $h_{11}^i + h_{22}^i - 2h_{33}^i = 0$ for all $i\in\{1,2,3\}$. On the other hand, the minimality condition implies $h_{11}^i + h_{22}^i + h_{33}^i = 0$, so we obtain $h_{11}^i + h_{22}^i=0$ and $h_{33}^i=0$ for all $i\in\{1,2,3\}$. Since we already have $h_{12}^3 = h_{33}^1 = h_{33}^2 = 0$ and $h_{11}^3 = h_{22}^3$, the only possibly non-zero components of the second fundamental form are $h_{11}^1 = -h_{22}^1$ and $h_{11}^2 = -h_{22}^2$. By \eqref{seccurv}, the sectional curvatures of the planes spanned by $\{e_1,e_2\}$, $\{e_1,e_3\}$ and $\{e_2,e_3\}$ are $K_{12} = \cos^2(\theta_1-\theta_2)-2(h_{11}^1)^2-2(h_{11}^2)^2$, $K_{13} = \cos^2(\theta_1-\theta_3)$ and $K_{23} = \cos^2(\theta_2-\theta_3)$. It follows from the latter two equalities that $\theta_1-\theta_3$ and $\theta_2-\theta_3$ are constant, and hence, by taking derivatives using \eqref{intcond1} and $h_{33}^i=0$, that $h_{11}^i = 0$ and $h_{22}^i = 0$. We conclude that the submanifold is totally geodesic, but by comparing $K_{12}$ and $K_{13}$, we then see that $y=0$, which is a contradiction.
	
	\textit{Case 3: None of the local functions $x$, $y$ and $z$ are identically zero.} We work on an open subset of $M^n$ where none of the functions vanish. It follows from \eqref{dim3eq1} and the minimality condition that there are local functions $\alpha_1$, $\alpha_2$ and $\alpha_3$ on this subset such that
	\begin{equation}\label{dim3eq2}
		\begin{aligned}
			& h^1_{22} = \alpha_1 z, & h^1_{33} = -\alpha_1 x, && h^1_{11} = \alpha_1(x-z), \\
			& h^2_{33} = \alpha_2 x, & h^2_{11} = -\alpha_2 y, && h^2_{22} = \alpha_2 (y-x),\\
			& h^3_{11} = \alpha_3 y, & h^3_{22} = -\alpha_3 z, && h^3_{33} = \alpha_3(z-y). \\
		\end{aligned}
	\end{equation}
	From these equations and \eqref{intcond1}, we obtain expressions for the derivatives of the components of $h$ in terms of the $\alpha_1$, $\alpha_2$, $\alpha_3$ and their derivatives. Substituting these into the Codazzi equation \eqref{codazzi} for $(X,Y,Z)=(e_1,e_2,e_3)$ and using $\theta_1+\theta_2+\theta_3=0 \mod \pi$ to eliminate $\theta_3$, we obtain
	\begin{equation*}
		\begin{aligned}
			& e_1(\alpha_3) = \frac{1}{4} \, \alpha_1 \, \alpha_3 \csc(2\theta_1+\theta_2) \Big( 5\cos(2\theta_1+\theta_2) - 7\cos(3\theta_2) + 2\cos(4\theta_1-\theta_2) \\ & \hspace{8cm} - \cos(4\theta_1+5\theta_2) + \cos(6\theta_1+3\theta_2) \Big), \\
			& e_2(\alpha_3) = \frac{1}{4} \, \alpha_2 \, \alpha_3 \csc(\theta_1+2\theta_2) \Big( 5\cos(\theta_1+2\theta_2) - 7 \cos(3\theta_1)+ 2\cos(\theta_1-4\theta_2)  \\ & \hspace{8cm} - \cos(5\theta_1+4\theta_2) + \cos(3\theta_1+6\theta_2) \Big).
		\end{aligned}
	\end{equation*}
	Similarly, for $(X,Y,Z)=(e_2,e_3,e_1)$ we obtain
	\begin{equation*}
		\begin{aligned}
			& e_2(\alpha_1) = -\frac{1}{4} \alpha_1 \, \alpha_2 \csc(\theta_1-\theta_2) \Big( 5\cos(\theta_1-\theta_2) - 7 \cos(3\theta_1+3\theta_2)+ 2 \cos(\theta_1+5\theta_2)  \\ & \hspace{8cm}  - \cos(5\theta_1+\theta_2)+ \cos(3\theta_1-3\theta_2) \Big),\\
			& e_3(\alpha_1)= -\frac{1}{4} \alpha_1 \, \alpha_3 \csc(2\theta_1+\theta_2) \Big( 5\cos(2\theta_1+\theta_2) -7\cos(3\theta_2) +2\cos(4\theta_1+5\theta_2)  \\ & \hspace{8cm} - \cos(4\theta_1-\theta_2)  + \cos(6\theta_1+3\theta_2) \Big)
		\end{aligned}
	\end{equation*}
	and for $(X,Y,Z)=(e_3,e_1,e_2)$ we obtain
	\begin{equation*}
		\begin{aligned}
			& e_1(\alpha_2) = \frac{1}{4} \alpha_1 \, \alpha_2 \csc(\theta_1-\theta_2) \Big( 5\cos(\theta_1-\theta_2) - 7 \cos(3\theta_1+3\theta_2) + 2 \cos(5\theta_1+\theta_2) \\ & \hspace{8cm} - \cos(\theta_1+5\theta_2) + \cos(3\theta_1-3\theta_2)\Big), \\
			& e_3(\alpha_2) = -\frac{1}{4} \alpha_2 \, \alpha_3 \csc(\theta_1+2\theta_2) \Big( 5\cos(\theta_1+2\theta_2) - 7\cos(3\theta_1)+ 2 \cos(5\theta_1+4\theta_2)\\ & \hspace{8cm} - \cos(\theta_1-4\theta_2)  + \cos(3\theta_1+6\theta_2)  \Big).
		\end{aligned}
	\end{equation*}
	Note that we are in the case $x y z\neq 0$, using these expressions above for the derivatives of $\alpha_1$, $\alpha_2$ and $\alpha_3$, the $Je_1$-component of the Codazzi equation for $(X,Y,Z)=(e_2,e_1,e_1)$ yields $\alpha_1\alpha_2=0$, the $Je_2$-component of the Codazzi equation for $(X,Y,Z)=(e_3,e_2,e_2)$ yields $\alpha_2\alpha_3=0$ and the $Je_3$-component of the Codazzi equation for $(X,Y,Z)=(e_1,e_3,e_3)$ yields $\alpha_1\alpha_3=0$. This implies that at least two of the functions $\{\alpha_1,\alpha_2,\alpha_3\}$ vanish identically. Because of the symmetry of the problem, we may assume without loss of generality that $\alpha_1=\alpha_2=0$.
	
	The $Je_2$-component of the Codazzi equation for $(X,Y,Z)=(e_1,e_2,e_1)$ then gives
	\begin{equation} \label{eqn=3.1}
		\alpha_3^2 = -2\csc(3\theta_1)\csc(3\theta_2)\cos(\theta_1-\theta_2),
	\end{equation}
	whereas the $Je_1$-component of the Codazzi equation for $(X,Y,Z)=(e_3,e_1,e_3)$ gives
	\begin{equation} \label{eqn=3.2}
		\begin{aligned}
			e_3(\alpha_3)=
			& \frac{1}{16} \csc(3\theta_1) \csc(\theta_1+2\theta_2) \csc(2\theta_1+\theta_2) [-32 \sin^2(2\theta_1+\theta_2)\cos(2\theta_1+\theta_2) \\
			& - \alpha_3^2[15\cos(2\theta_1+\theta_2) - \cos(8\theta_1+7\theta_2) + \cos(8\theta_1+\theta_2) + 4 \cos(6\theta_1+3\theta_2) \\
			& - \cos(10\theta_1+5\theta_2) + 4 \cos(6\theta_1-3\theta_2)  - 16 \cos(4\theta_1-\theta_2) + \cos(2\theta_1-5\theta_2) \\	
			& - 8 \cos(3\theta_2) + \cos(2\theta_1+7\theta_2)]].
		\end{aligned}
	\end{equation}
	Substituting \eqref{eqn=3.1} and \eqref{eqn=3.2} into the Codazzi equation for $(X,Y,Z)=(e_3,e_2,e_3)$ gives
	\begin{equation} \label{5.13}
		-5\cos(\theta_1-\theta_2) + 2\cos(3(\theta_1-\theta_2)) + (1+2\cos(2(\theta_1-\theta_2)))\cos(3(\theta_1+\theta_2))=0. \end{equation}
	On the other hand, the difference between the Gauss equations for $(X,Y,Z,W)=(e_1,e_2,e_2,e_1)$ and $(X,Y,Z,W)=(e_1,e_3,e_3,e_1)$ gives the following equation for $\alpha_3$.
	\begin{equation} \label{5.14} 1+\alpha_3^2\big(\cos(2(\theta_1-\theta_2))-\cos(\theta_1-\theta_2)\cos(3(\theta_1+\theta_2)\big)=0. \end{equation}
	Substituting \eqref{eqn=3.1} into  \eqref{5.14}, we obtain
	\begin{equation} \label{5.15}
		-2\cos(\theta_1-\theta_2) -\cos(3(\theta_1-\theta_2)) + (1+2\cos(2(\theta_1-\theta_2)))\cos(3(\theta_1+\theta_2))=0. \end{equation}
	By combining \eqref{5.13} and \eqref{5.15}, we obtain that all angle functions are constant. Since they are mutually different  modulo $\pi$, it follows from Theorem \ref{theoCAdim3} that $f$, restricted to the open subset of $M^n$ on which we are working, is the Gauss map of a part of a tube around a Veronese surface in $S^4(1)$, which contradicts $h_{12}^3=0$.
\end{proof}

\subsection{Classification in dimension $n=4$}
The following proposition shows that, for $n=4$, no new examples of minimal Lagrangian submanifolds of $Q^4$ with constant sectional curvature occur.

\begin{proposition} \label{propCSCn=4}
	Let $f:M^4 \to Q^4$ be a minimal Lagrangian immersion such that $M^4$ has constant sectional curvature. Then $M^4$ has constant sectional curvature $2$ and $f$ is the Gauss map of the standard embedding $S^4(r) \to S^5(1)$.
\end{proposition}

\begin{proof}
	Choose $A \in \mathcal A$ as in Example \ref{ex2}. By Theorem \ref{theoTG}, it suffices to show that $f$ is a totally geodesic immersion. We know from Proposition \ref{propCSC2} that $h_{ij}^k=0$ for all mutually different indices $i$, $j$ and $k$, so we only have to show that $h^i_{jj}=0$ for all $i,j \in \{1,2,3,4\}$. We will proceed by contradiction and distinguish three cases.
	
	\smallskip
	
	\emph{Case 1: There are at least three different indices $i$ for which $h^i_{11}$, $h^i_{22}$, $h^i_{33}$ and $h^i_{44}$ are not all zero.} Without loss of generality, we may assume that these three indices are $1$, $2$ and $3$. Remark that if $h^i_{ii} \neq 0$, there exists also an index $j \neq i$ such that $h^i_{jj} \neq 0$ due to minimality. For every $i \in \{1,2,3\}$ we consider the following system of equations coming from \eqref{CSCeq1}:
	\begin{equation} \label{eqn=4.1}
		\left\{ \begin{array}{l}
			\sin(\theta_j-\theta_i) \sin(\theta_j+\theta_i-2\theta_k) h_{jj}^i - \sin(\theta_k-\theta_i) \sin(\theta_k+\theta_i-2\theta_j) h_{kk}^i = 0, \\
			\sin(\theta_k-\theta_i) \sin(\theta_k+\theta_i-2\theta_{\ell}) h_{kk}^i - \sin(\theta_{\ell}-\theta_i) \sin(\theta_{\ell}+\theta_i-2\theta_k) h_{\ell\ell}^i = 0, \\
			\sin(\theta_{\ell}-\theta_i) \sin(\theta_{\ell}+\theta_i-2\theta_j) h_{\ell\ell}^i - \sin(\theta_j-\theta_i) \sin(\theta_j+\theta_i-2\theta_{\ell}) h_{jj}^i = 0,
		\end{array} \right.
	\end{equation}
	where $\{j,k,\ell\}=\{1,2,3,4\}\setminus\{i\}$. By our assumption, the determinant of this system of linear equations in $h^i_{jj}$, $h^i_{kk}$ and $h^i_{\ell\ell}$ must vanish. A straightforward computation shows that this determinant is
	\begin{multline} \label{det}  2\sin(\theta_j-\theta_i)\sin(\theta_k-\theta_i)\sin(\theta_{\ell}-\theta_i)\sin(\theta_k-\theta_j)\sin(\theta_{\ell}-\theta_j)\sin(\theta_{\ell}-\theta_k)\\
		(\cos(\theta_i+\theta_j-\theta_k-\theta_{\ell})+\cos(\theta_i-\theta_j+\theta_k-\theta_{\ell})+\cos(\theta_i-\theta_j-\theta_k+\theta_{\ell})).
	\end{multline}
	Since all angle functions are different modulo $\pi$ by Proposition \ref{propCSC1}, we obtain
	\begin{equation} \label{eqn=4.2}
		\cos(\theta_i+\theta_j-\theta_k-\theta_{\ell})+\cos(\theta_i-\theta_j+\theta_k-\theta_{\ell})+\cos(\theta_i-\theta_j-\theta_k+\theta_{\ell}) = 0.
	\end{equation}
	Taking the derivative of \eqref{eqn=4.2} in the direction of $e_i$, using \eqref{intcond1}, Corollary \ref{cor1}, the minimality condition and some elementary trigonometric identities, yields
	\begin{equation} \label{eqn=4.3}
		\sin(\theta_i-\theta_j)\cos(\theta_k-\theta_\ell)h_{jj}^i + \sin(\theta_i-\theta_k)\cos(\theta_j-\theta_{\ell})h_{kk}^i + \sin(\theta_i-\theta_{\ell})\cos(\theta_j-\theta_k)h_{\ell\ell}^i = 0.
	\end{equation}
	This is another linear equation in $h_{jj}^i$, $h_{kk}^i$ and $h_{\ell\ell}^i$ and the determinant of the system formed by any two equations from \eqref{eqn=4.1} and \eqref{eqn=4.3} must be zero. We will denote the determinant of \eqref{eqn=4.3} and the first two equations from \eqref{eqn=4.1} (which both involve $h_{kk}^i$) by $\Delta^i_{kk}$.  A straightforward computation shows that the equations \begin{equation*}
		\begin{aligned}
			&(\Delta^1_{33}+\Delta^1_{44})/(\sin{(\theta_1-\theta_3)}\sin{(\theta_1-\theta_4)})+(\Delta^2_{33}+\Delta^2_{44})/(\sin{(\theta_2-\theta_3)}\sin{(\theta_2-\theta_4)})=0,\\ &(\Delta^2_{11}+\Delta^2_{44})/(\sin{(\theta_2-\theta_1)}\sin{(\theta_2-\theta_4)})+(\Delta^3_{11}+\Delta^3_{44})/(\sin{(\theta_3-\theta_1)}\sin{(\theta_3-\theta_4)})=0,\\ &(\Delta^3_{22}+\Delta^3_{44})/(\sin{(\theta_3-\theta_2)}\sin{(\theta_3-\theta_4)})+(\Delta^1_{22}+\Delta^1_{44})/(\sin{(\theta_1-\theta_2)}\sin{(\theta_1-\theta_4)})=0,
		\end{aligned}
	\end{equation*}
	are equivalent to
	\begin{align}
		& \sin(\theta_1+\theta_2-\theta_3-\theta_4)(\cos(\theta_4-\theta_3)+\cos(3(\theta_4-\theta_3))) = 0, \nonumber \\
		& \sin(\theta_2+\theta_3-\theta_1-\theta_4)(\cos(\theta_4-\theta_1)+\cos(3(\theta_4-\theta_1))) = 0, \label{eqn=4.4} \\
		& \sin(\theta_3+\theta_1-\theta_2-\theta_4)(\cos(\theta_4-\theta_2)+\cos(3(\theta_4-\theta_2))) = 0, \nonumber
	\end{align}
	respectively. We distinguish two subcases.
	
	\smallskip
	
	\emph{Case 1.1: $\sin(\theta_1+\theta_2-\theta_3-\theta_4)\sin(\theta_2+\theta_3-\theta_1-\theta_4)\sin(\theta_3+\theta_1-\theta_2-\theta_4)=0$.} Assume that $\sin(\theta_1+\theta_2-\theta_3-\theta_4)=0$, the other two cases are analogous. Together with $\theta_1+\theta_2+\theta_3+\theta_4=0 \mod \pi$, we obtain $\theta_1+\theta_2=0 \mod \pi$ and $\theta_3+\theta_4=0 \mod \pi$, so that \eqref{eqn=4.2} is equivalent to
	\begin{equation} \label{eqn=4.5}
		1 + 2 \cos(2\theta_1)\cos(2\theta_3) = 0.
	\end{equation}
	Deriving $\theta_3+\theta_4=0 \mod \pi$ in the direction of $e_1$ gives $h^1_{33}+h^1_{44}=0$, so it follows from the equation involving $h^1_{33}$ and $h^1_{44}$ in \eqref{eqn=4.1} for $i=1$ that $(\cos(2\theta_1)\cos(2\theta_3)-\cos(4\theta_3))h^1_{33}=0$, which, in combination with \eqref{eqn=4.5}, yields $(1 + 2\cos(4\theta_3))h^1_{33}=0$.
	
	If $1 + 2\cos(4\theta_3)=0$, it follows from \eqref{eqn=4.5}, $\theta_1+\theta_2=0 \mod \pi$ and $\theta_3+\theta_4=0 \mod \pi$ that all local angle functions are constant. But from \eqref{intcond1} we then obtain that all $h^i_{jj}$ are zero, which contradicts the assumption that we made for Case 1.
	
	If, on the other hand, $h^1_{33}=0$, and hence also $h^1_{44}=0$, the assumption for Case 1 yields $h^1_{22} \neq 0$ and the equations involving $h^1_{22}$ in \eqref{eqn=4.1} for $i=1$ become $\sin(\theta_1+\theta_2-2\theta_3) = \sin(\theta_1+\theta_2-2\theta_4) = 0$. Since $\theta_1+\theta_2=0 \mod \pi$ and $\theta_3+\theta_4=0 \mod \pi$, both equations reduce to $\sin(2\theta_3)=0$. Again, from \eqref{eqn=4.5}, $\theta_1+\theta_2=0 \mod \pi$ and $\theta_3+\theta_4=0 \mod \pi$, we obtain that all local angle functions are constant, which is a contradiction.
	
	\smallskip
	
	\emph{Case 1.2: $\sin(\theta_1+\theta_2-\theta_3-\theta_4)\sin(\theta_2+\theta_3-\theta_1-\theta_4)\sin(\theta_3+\theta_1-\theta_2-\theta_4) \neq 0$.} It follows from \eqref{eqn=4.4} that there exist $k_1,k_2,k_3 \in \mathbb Z \setminus 4\mathbb Z$ such that $\theta_1 = \theta_4 + k_1\pi/4$, $\theta_2 = \theta_4 + k_2\pi/4$ and $\theta_3 = \theta_4 + k_3\pi/4$. By combining this with $\theta_1+\theta_2+\theta_3+\theta_4=0 \mod\pi$ we obtain that all local angle functions are constant, which is a contradiction as before.
	
	\smallskip
	
	\emph{Case 2: There are exactly two different indices $i$ for which $h^i_{11}$, $h^i_{22}$, $h^i_{33}$ and $h^i_{44}$ are not all zero.} Without loss of generality, we may assume that these indices are $1$ and $2$. As before, we can then obtain the first equation of \eqref{eqn=4.4}. If $\sin(\theta_1+\theta_2-\theta_3-\theta_4)=0$, we proceed as in Case~1.1 to obtain a contradiction. If $\cos(\theta_4-\theta_3)+\cos(3(\theta_4-\theta_3))=0$, then $\theta_4 = \theta_3 + k\pi/4$ for some $k \in \mathbb Z \setminus 4\mathbb Z$. Deriving this equation in the direction of $e_i$ gives $h^i_{33}=h^i_{44}$ for all $i\in\{1,2,3,4\}$ and the equations involving $h^1_{33}$ and $h^1_{44}$ in \eqref{eqn=4.1} for $i=1$, respectively $h^2_{33}$ and $h^2_{44}$ in \eqref{eqn=4.1} for $i=2$, reduce to
	\begin{equation} \label{eqn=4.6}
		\begin{aligned}
			& \sin(2\theta_3-2\theta_1+k\pi/4)h^1_{33}=0, \\
			& \sin(2\theta_3-2\theta_2+k\pi/4)h^2_{33}=0.
		\end{aligned}
	\end{equation}
	We again distinguish two subcases.
	
	\smallskip
	
	\emph{Case 2.1: $\sin(2\theta_3-2\theta_1+k\pi/4)\sin(2\theta_3-2\theta_2+k\pi/4)=0$.} Assume that $\sin(2\theta_3-2\theta_1+k\pi/4)=0$, the other case is analogous. Then $\theta_3=\theta_1-k\pi/8+\ell\pi/2$, and hence $\theta_4=\theta_1+k\pi/8+\ell\pi/2$, for some $\ell \in \mathbb Z$ and $k \in \mathbb Z \setminus 4\mathbb Z$.  Equation \eqref{eqn=4.2} is equivalent to $\cos(\theta_2-\theta_1)=0$, such that $\theta_2=\theta_1+\pi/2+m\pi$ for some $m \in \mathbb Z$. This means that all local angle functions can be written as $\theta_1$ plus a constant and it follows from $\theta_1+\theta_2+\theta_3+\theta_4=0 \mod\pi$ that they are all constant, which gives the desired contradiction.
	
	\emph{Case 2.2: $\sin(2\theta_3-2\theta_1+k\pi/4)\sin(2\theta_3-2\theta_2+k\pi/4) \neq 0$.} It then follows from \eqref{eqn=4.6} that $h^1_{33}=h^2_{33}=0$ and hence also $h^1_{44}=h^2_{44}=0$. By the assumption that we made for Case~2, $h^1_{22} \neq 0$. From the equations involving $h^1_{22}$ in \eqref{eqn=4.1} for $i=1$ we have $\sin(\theta_1+\theta_2-2\theta_3)=\sin(\theta_1+\theta_2-2\theta_4)=0$. If $k$ is odd, this is a contradiction. If $k$ is even, we can conclude from $\sin(\theta_1+\theta_2-2\theta_3)=0$, $\theta_4=\theta_3+k\pi/4$ and $\theta_1+\theta_2+\theta_3+\theta_4=0 \mod\pi$ that $\theta_1+\theta_2$, $\theta_3$ and $\theta_4$ are constant. To finish the proof in this case, we compute the sectional curvature of the plane spanned by $e_1$ and $e_3$ using \eqref{seccurv}. Since $h^1_{33}=h^2_{33}=h^3_{11}=h^4_{11}=h^3_{12}=h^4_{13}=0$, we obtain that $K_{13}=2\cos^2(\theta_3-\theta_1)$ is constant. Again, we see that all the local angle functions are constant, which is a contradiction.
	
	\smallskip
	
	\emph{Case 3: There is exactly one index $i$ for which $h^i_{11}$, $h^i_{22}$, $h^i_{33}$ and $h^i_{44}$ are not all zero.} Without loss of generality, we may assume that $i=1$ and hence $h^2_{jj}=h^3_{jj}=h^4_{jj}=0$ for all $j$. Denote by $c$ the constant sectional curvature of $M^n$. A straightforward computation of the sectional curvatures using the definition of the curvature tensor and \eqref{intcond2} gives
	\begin{align}
		& c = K_{23} = -h^1_{22} h^1_{33} \cot(\theta_2-\theta_1)  \cot(\theta_3-\theta_1), \nonumber \\
		& c = K_{24} = -h^1_{22} h^1_{44} \cot(\theta_2-\theta_1)  \cot(\theta_4-\theta_1), \label{eqn=4.7} \\
		& c = K_{34} = -h^1_{33} h^1_{44} \cot(\theta_3-\theta_1)  \cot(\theta_4-\theta_1). \nonumber
	\end{align}
	On the other hand, from \eqref{seccurv}, it follows that
	\begin{align}
		& c = K_{23} = 2 \cos^2(\theta_3-\theta_2) + h^1_{22}h^1_{33}, \nonumber \\
		& c = K_{24} = 2 \cos^2(\theta_4-\theta_2) + h^1_{22}h^1_{44}, \label{eqn=4.8} \\
		& c = K_{34} = 2 \cos^2(\theta_4-\theta_3) + h^1_{33}h^1_{44}. \nonumber
	\end{align}
	
	Remark that, for a fixed $i \in \{1,2,3,4\}$, at most one of the functions $\cos(\theta_i-\theta_j)$, with $j \in \{1,2,3,4\} \setminus \{i\}$, can be zero since all the local angle functions are mutually different modulo $\pi$ by Proposition \ref{propCSC1}. In particular, this implies that $c \neq 0$. Indeed, for $c=0$, \eqref{eqn=4.7} would imply that $h^1_{jj}=0$ for at least one $j \in \{2,3,4\}$ and then \eqref{eqn=4.8} would imply that at least two of the functions $\cos(\theta_3-\theta_2)$, $\cos(\theta_4-\theta_3)$ and $\cos(\theta_4-\theta_2)$ are zero, which is impossible.  As $c\neq0$, we obtain from \eqref{eqn=4.7}  that $d:= h^1_{22} \cot(\theta_2-\theta_1) = h^1_{33} \cot(\theta_3-\theta_1) = h^1_{44} \cot(\theta_4-\theta_1)$ is a constant satisfying $c=-d^2$.
	
	Without loss of generality, we will assume from now on that $\cos(\theta_3-\theta_2)\cos(\theta_4-\theta_2) \neq 0$. From $c \neq 0$ and \eqref{eqn=4.7}, we have that none of the functions $\cos(\theta_j-\theta_1)$ with $j\in\{2,3,4\}$ is zero. When putting $h^1_{jj} = d/\cot(\theta_j-\theta_1)$ in the first two equations of \eqref{eqn=4.8}, using the assumption $\cos(\theta_3-\theta_2)\cos(\theta_4-\theta_2) \neq 0$, we obtain
	\begin{align*}
		c = 2 \cos(\theta_3-\theta_2) \cos(\theta_2-\theta_1) \cos(\theta_3-\theta_1), \\
		c = 2 \cos(\theta_4-\theta_2) \cos(\theta_2-\theta_1) \cos(\theta_4-\theta_1),
	\end{align*}
	from which we get that $\cos(\theta_3-\theta_2)\cos(\theta_3-\theta_1) = \cos(\theta_4-\theta_2)\cos(\theta_4-\theta_1)$ or, equivalently, $\sin(\theta_1+\theta_2-\theta_3-\theta_4)=0$. We can now proceed as in Case 1.1 to obtain a contradiction.
\end{proof}

\subsection{Classification in dimension $n\geq 5$}

The following proposition shows that, also for $n \geq 5$, no new examples of minimal Lagrangian submanifolds of $Q^n$ with constant sectional curvature occur.

\begin{proposition} \label{propCSCn>4}
	For $n \geq 5$, let $f:M^n \to Q^n$ be a minimal Lagrangian immersion such that $M^n$ has constant sectional curvature. Then $M^n$ has constant sectional curvature $2$ and $f$ is the Gauss map of the standard embedding $S^n(r) \to S^{n+1}(1)$.
\end{proposition}

\begin{proof}
	Choose $A \in \mathcal A$ as in Example \ref{ex2}. By Proposition \ref{propCSC1}, we know that the angle functions are either all the same modulo $\pi$ or all mutually different modulo $\pi$ and that in the former case, the immersion is totally geodesic. Hence, assume that all angle functions are mutually different modulo $\pi$. We know from Proposition \ref{propCSC2} that all components $h_{ij}^k$ of the second fundamental form, where $i$, $j$ and $k$ are mutually different, vanish. Hence, it suffices to show that the components of the second fundamental form for which at least two indices are the same also vanish.
	
	\textit{Step 1: For any fixed $i$, if there exists a $j \neq i$ such that $h_{jj}^i=0$, then $h_{jj}^i=0$ for all $j$.} Without loss of generality, we assume that $i=1$ and $j=2$, so $h_{22}^1=0$. We will prove the claim by contradiction, so suppose that $h_{jj}^1 \neq 0$ for some $j$. Remark that if $h_{11}^1 \neq 0$, there will be a $j \neq 1,2$ for which $h_{jj}^1 \neq 0$. Indeed, this follows from the fact that $h_{11}^1 + h_{22}^1 + h_{33}^1 + \cdots + h_{nn}^1 = 0$ due to minimality. For simplicity, we can hence assume that $h_{33}^1 \neq 0$. From \eqref{CSCeq1} for $i \geq 3$, $j=2$ and $k=1$, we obtain $h_{ii}^1 \sin(\theta_i+\theta_1-2\theta_2) = 0$. For $i=3$, this implies that $\theta_3+\theta_1-2\theta_2 = 0 \mod\pi$ and for $i \geq 4$, this implies that $h_{44}^1 = h_{55}^1 = \cdots h_{nn}^1 = 0$, since all the angle functions are mutually different modulo $\pi$. A similar argument, but now for $j=4$, respectively $j=5$, instead of $j=2$, yields $\theta_3+\theta_1-2\theta_4 = 0 \mod\pi$, respectively $\theta_3+\theta_1-2\theta_5 = 0 \mod\pi$. Summarizing the relations between the angle functions, we have that $2\theta_2$, $2\theta_4$ and $2\theta_5$ are all equal to $\theta_1+\theta_3$ modulo $\pi$. This means that at least two of the angle functions $\theta_2$, $\theta_4$ and $\theta_5$ must be equal modulo $\pi$, which is a contradiction.
	
	\textit{Step 2.  For any fixed $i$, there exists a $j \neq i$ such that $h_{jj}^i=0$.} It follows from \eqref{CSCeq1} that for any choice of mutually different indices $j$, $k$ and $\ell$, all different from $i$, the variables $h_{jj}^i$, $h_{kk}^i$ and $h_{\ell\ell}^i$ satisfy the system \eqref{eqn=4.1}.
	It suffices to show that for some choice of $j$, $k$ and $\ell$, the determinant of this system is non zero. Indeed, this will imply that for that particular choice of $j$, $k$ and $\ell$, one has $h_{jj}^i = h_{kk}^i = h_{\ell\ell}^i = 0$. The determinant of the system is given by \eqref{det} and it is clear that the only factor which could possibly be zero is the last one. We assume that this factor is zero for all choices of mutually different $j$, $k$ and $\ell$, different from $i$, and we will prove a contradiction. Define the function $f_1(\theta):=\cos(\theta_i+\theta_j-\theta_k-\theta)+\cos(\theta_i-\theta_j+\theta_k-\theta)+\cos(\theta_i-\theta_j-\theta_k+\theta)$. Since it satisfies the differential equation $f_1''+f_1=0$, we can write it as $f_1(\theta)=a_1\sin(\theta+b_1)$ for some constants $a_1$ and $b_1$ depending on $\theta_i$, $\theta_j$ and $\theta_k$. By our assumption, $f_1(\theta)=0$ for at least two values of $\theta$ which are different modulo $\pi$. This implies that $a_1=0$ and hence that $f_1=0$ identically. In particular, $f_1(0)=\cos(\theta_i+\theta_j-\theta_k)+\cos(\theta_i-\theta_j+\theta_k)+\cos(\theta_i-\theta_j-\theta_k)=0$ for all mutually different $j$ and $k$, different from $i$. We can basically repeat the same argument by defining the function
	$f_2(\theta):=\cos(\theta_i+\theta_j-\theta)+\cos(\theta_i-\theta_j+\theta)+\cos(\theta_i-\theta_j-\theta)$. Since $f_2''+f_2=0$, we have $f_2(\theta)=a_2\sin(\theta+b_2)$ for some constants $a_2$ and $b_2$ depending on $\theta_i$ and $\theta_j$. Since $f_2(\theta)=0$ for at least two values of $\theta$ which are different modulo $\pi$, we conclude that $a_2=0$ and hence $f_2=0$ identically. Since $f_2(\theta)=(\cos(\theta_i+\theta_j)+2\cos(\theta_i-\theta_j))\cos\theta+\sin(\theta_i+\theta_j)\sin\theta$, we obtain that $\sin(\theta_i+\theta_j)=0$ for all $j$ different from $i$ modulo $\pi$. This implies that all angles $\theta_j$, with $j \neq i$, are equal modulo $\pi$, a contradiction.
\end{proof}

\subsection{Proof of Theorem~\ref{theoCSCconclusion}} Combining the results in Sections 5.2-5.5, we obtain the classification result in  Theorem~\ref{theoCSCconclusion}.

\section{Minimal Lagrangian submanifolds of $Q^n$ with $n-1$ equal angle functions}

In this section, we give  the proof of Theorem  \ref{theoAllButOneSameAngles}, which is a classification theorem for minimal Lagrangian submanifolds of $Q^n$ with $n-1$ equal angle functions.
We use the same notations as in Section 3. Case (i) and Case (ii) follow immediately from Corollary \ref{cor2}. In Case (iii), we have that $\alpha$ is not a constant modulo $\pi$. We first show that $M^n$ must be a warped product $I\times_\rho S^{n-1}(1)$ with $\rho(\alpha)=|c_1(\sin n\alpha)^{-\frac{1}{n}}|$ for some positive constant $c_1$, and the angle function $\alpha$ satisfies the first order  ordinary differential equation \eqref{6.1}. First,  as $\theta_1=(n-1)\alpha \mod\pi$ and $\theta_2=\cdots=\theta_n=-\alpha \mod\pi$, we obtain that the one-form $s$ vanishes on $TM^n$ from Corollary \ref{cor1}. It follows from \eqref{intcond2} that $h_{ij}^k=0$ for any $i$ and any mutually different $j,k\in\{2,\ldots,n\}$. Using \eqref{intcond1}, for any $i\in\{2,\ldots,n\}$,  as $n\geq 3$, we have that $h_{ii}^i=e_i(\theta_i)=e_i(\theta_k)=h_{kk}^i=0$ and $h_{11}^i=e_i(\theta_1)=-(n-1)e_i(\theta_k)=-(n-1)h_{kk}^i=0$ for any $k\in\{2,\ldots,n\}$, different from $i$. Therefore, we obtain that the only non-zero components of the second fundamental form are $h_{11}^1=(n-1)e_1(\alpha)$ and $h_{22}^1=\cdots=h_{nn}^1=-e_1(\alpha)$ and $e_k(\alpha)=0$ for any $k\in\{2,\ldots,n\}$.

By applying \eqref{intcond2} again, we obtain $g(\nabla_{e_1}e_1,e_k)=\omega_1^k(e_1)=0$ for any $k\in\{2,\ldots,n\}$, which means that the distribution spanned by $e_1$ is autoparallel. We also have $g(\nabla_{e_k}e_j,e_1)=\omega_j^1(e_k)=-\cot{(n\alpha)}h_{jk}^1=\cot{(n\alpha)}e_1(\alpha)\delta_{jk}$ for any mutually different $j,k\in\{2,\ldots,n\}$, and since $e_k(\alpha)=0$ for any $k\in\{2,\ldots,n\}$, we obtain that the distribution spanned by $\{e_2,\ldots,e_n\}$ is spherical. Hence, by applying a theorem of Hiepko \cite{Hiepko} (see also a general result in \cite{Nolker}), we conclude that $M^n$ is a warped product $I\times_\rho N^{n-1}$ and the warping function $\rho$ satisfies that $\frac{e_1(\rho)}{\rho}=-\cot{(n\alpha)}e_1(\alpha)$. Hence, $\rho=|c_1(\sin n\alpha)^{-\frac{1}{n}}|$ for some positive constant $c_1$.

We can further show that $N^{n-1}$ has positive constant sectional curvature. In fact, the sectional curvature $K^N$ of the plane spanned by $e_j$ and $e_k$, for mutually different $j,k\in\{2,\ldots,n\}$, can be calculated as follows:
\begin{equation}\label{6.4}
	\begin{aligned}
		K^N_{jk}&=\rho^2(K_{jk}+(\frac{e_1(\rho)}{\rho})^2)\\
		&=\rho^2(2\cos^2{(\theta_j-\theta_k)}+e_1(\theta_j)e_1(\theta_k)+(\frac{e_1(\rho)}{\rho})^2)\\
		&=\rho^2(2+e_1^2(\alpha)+(-\cot{(n\alpha)}e_1(\alpha))^2)\\
		&=(c_1(\sin (n\alpha))^{-\frac{1}{n}})^2(2+e_1^2(\alpha)(\sin{n\alpha})^{-2}),
	\end{aligned}
\end{equation}
where we used \eqref{seccurv} in the second equality.
For any $k\in\{2,\ldots,n\}$, by considering the  $Je_k$-component of the Codazzi equation for $(X,Y,Z)=(e_k,e_1,e_1)$, we get that
\begin{equation}\label{6.5}
	e_1(e_1(\alpha))-(n+1)\cot{(n\alpha)}(e_1(\alpha))^2-\sin{(2n\alpha)}=0.
\end{equation}
We take the derivative of $K^N_{jk}$ with respect to $e_1$ and obtain that
\begin{equation*}
	e_1(K^N_{jk})=2c_1^2(\sin (n\alpha))^{-\frac{2}{n}-2}e_1(\alpha)\big(e_1(e_1(\alpha))-(n+1)\cot{(n\alpha)}(e_1(\alpha))^2-\sin{(2n\alpha)}\big)=0,
\end{equation*}
which means that $N^{n-1}$ has constant sectional curvature $c_0$. From the expression \eqref{6.4}, we know that $c_0$ is a positive constant and we can always choose $c_1$ such that $c_0=1$.  Therefore, we have proved that $M^n$ is a warped product $I\times_\rho S^{n-1}(1)$ with $\rho(\alpha)=|c_1(\sin n\alpha)^{-\frac{1}{n}}|$ for some positive constant $c_1$, and the angle function $\alpha$ satisfies the first order ordinary differential equation \eqref{6.1}.

Secondly, we show that  $f$ is locally isometric to the Gauss map of a rotational hypersurface of $S^{n+1}(1)$ with the profile curve $\gamma(\theta)\subset S^2(1)$ given by \eqref{6.2}.
As $\theta_1=(n-1)\alpha \mod\pi$ and $\theta_2=\cdots=\theta_n=-\alpha \mod\pi$, from the proof of Theorem \ref{theo1}, we know that $M$ is the Gauss map of a  hypersurface of $S^{n+1}(1)$ with principal curvatures given by
\begin{equation}\label{6.6}
	\lambda_1=\cot{(\theta_1)}=\cot{((n-1)\alpha)}\neq\lambda_2=\cdots=\lambda_n=\cot{(\theta_2)}=-\cot{(\alpha)},
\end{equation}
which immediately implies that $M$ is the Gauss map of a  rotational hypersurface of $S^{n+1}(1)$  by applying a theorem of do Carmo and Dajczer (see  Theorem 4.2 in \cite{CD1983}).
In order to finish the proof of Case (iii), we only need to check that the principal curvatures of a rotational hypersurface of $S^{n+1}(1)$ with the profile curve $\gamma(\theta)\subset S^2(1)$ given by \eqref{6.2} and  $\alpha$ satisfying \eqref{6.3} are the same as that in \eqref{6.6}. This can be verified straightforwardly by applying the formulas of the principal curvatures computed in \cite{LiMaWei}, we omit the details here.

Finally, we note that  the differential equation \eqref{6.3} is equivalent to \eqref{6.5}.
If we define a new parameter $s$ by $\frac{d s}{d\theta}=-\frac{1}{\sqrt{2}}\sqrt{1-(\frac{d\alpha}{d\theta})^2}(\sin{(n\alpha)})^{-1}$,
then $\alpha$ satisfies the following differential equation with respect to the new parameter $s$:
\begin{equation}
	\frac{d^2\alpha}{ds^2}-(n+1)\cot{(n\alpha)}(\frac{d\alpha}{ds})^2-\sin{(2n\alpha)}=0.
\end{equation}
We can also calculate the length of $\frac{d}{d s}$ directly by using the induced metric of  the Gauss map and find that $\frac{d}{d s}$  is a unit vector. Hence, $\frac{d}{d s}=\pm e_1$, which implies the equivalence of \eqref{6.3} and \eqref{6.5}.
This completes the proof of Case (iii). Therefore, we finish the proof of Theorem  \ref{theoAllButOneSameAngles}.




\begin{thebibliography}{10}
	
	\bibitem{C}
	\'{E}lie Cartan, \emph{Sur quelques familles remarquables d'hypersurfaces},
	C.R. Congr\`es Math. Li\`ege, 1939, 30--41 (see also: Oeuvres Compl\`etes:
	Partie III, Vol. 2, 1481--1492).
	
	\bibitem{CU}
	Ildefonso Castro and Francisco Urbano, \emph{Minimal {L}agrangian surfaces in
		{$\Bbb S^2\times\Bbb S^2$}}, Comm. Anal. Geom. \textbf{15} (2007), no.~2,
	217--248. 
	
	\bibitem{Chen2001}
	Bang-Yen Chen, \emph{Riemannian geometry of {L}agrangian submanifolds},
	Taiwanese J. Math. \textbf{5} (2001), no.~4, 681--723. 
	
	\bibitem{Chen2011}
	\bysame, \emph{Pseudo-{R}iemannian geometry, {$\delta$}-invariants and
		applications}, World Scientific Publishing Co. Pte. Ltd., Hackensack, NJ,
	2011, With a foreword by Leopold Verstraelen. 
	
	\bibitem{CO1974}
	Bang-yen Chen and Koichi Ogiue, \emph{On totally real submanifolds}, Trans.
	Amer. Math. Soc. \textbf{193} (1974), 257--266. 
	
	\bibitem{CD1983}
	Manfredo. do~Carmo and Marcos. Dajczer, \emph{Rotation hypersurfaces in spaces of constant
		curvature}, Trans. Amer. Math. Soc. \textbf{277} (1983), no.~2, 685--709.
	
	\bibitem{Ejiri1982}
	Norio Ejiri, \emph{Totally real minimal immersions of {$n$}-dimensional real
		space forms into {$n$}-dimensional complex space forms}, Proc. Amer. Math.
	Soc. \textbf{84} (1982), no.~2, 243--246. 
	
\bibitem{GT}
Jianquan Ge and Zizhou Tang, \emph{Geometry of isoparametric hypersurfaces in
  {R}iemannian manifolds}, Asian J. Math. \textbf{18} (2014), no.~1, 117--125.

	\bibitem{Hiepko}
	S\"{o}nke Hiepko, \emph{Eine innere {K}ennzeichnung der verzerrten {P}rodukte},
	Math. Ann. \textbf{241} (1979), no.~3, 209--215. 
	
	\bibitem{LiMaWei}
	Haizhong Li, Hui Ma, and Guoxin Wei, \emph{A class of minimal {L}agrangian
		submanifolds in complex hyperquadrics}, Geom. Dedicata \textbf{158} (2012),
	137--148. 
	
	\bibitem{MaOhnita1}
	Hui Ma and Yoshihiro Ohnita, \emph{On {L}agrangian submanifolds in complex
		hyperquadrics and isoparametric hypersurfaces in spheres}, Math. Z.
	\textbf{261} (2009), no.~4, 749--785. 
	
	\bibitem{MaOhnita2}
	\bysame, \emph{Hamiltonian stability of the {G}auss images of homogeneous
		isoparametric hypersurfaces. {I}}, J. Differential Geom. \textbf{97} (2014),
	no.~2, 275--348. 
	
	\bibitem{MaOhnita3}
	\bysame, \emph{Hamiltonian stability of the {G}auss images of homogeneous
		isoparametric hypersurfaces {II}}, Tohoku Math. J. (2) \textbf{67} (2015),
	no.~2, 195--246. 
	
	\bibitem{MunznerI}
	Hans~Friedrich M\"{u}nzner, \emph{Isoparametrische {H}yperfl\"{a}chen in
		{S}ph\"{a}ren}, Math. Ann. \textbf{251} (1980), no.~1, 57--71. 
	
	\bibitem{MunznerII}
	\bysame, \emph{Isoparametrische {H}yperfl\"{a}chen in {S}ph\"{a}ren. {II}.
		\"{U}ber die {Z}erlegung der {S}ph\"{a}re in {B}allb\"{u}ndel}, Math. Ann.
	\textbf{256} (1981), no.~2, 215--232. 
	
	\bibitem{Nolker}
	Stefan N\"{o}lker, \emph{Isometric immersions of warped products}, Differential
	Geom. Appl. \textbf{6} (1996), no.~1, 1--30. 
	
	\bibitem{Oh1990}
	Yong-Geun Oh, \emph{Second variation and stabilities of minimal {L}agrangian
		submanifolds in {K}\"{a}hler manifolds}, Invent. Math. \textbf{101} (1990),
	no.~2, 501--519. 
	
	\bibitem{Oh1993}
	\bysame, \emph{Volume minimization of {L}agrangian submanifolds under
		{H}amiltonian deformations}, Math. Z. \textbf{212} (1993), no.~2, 175--192.
	
	\bibitem{Palmer1994}
	Bennett Palmer, \emph{Buckling eigenvalues, {G}auss maps and {L}agrangian
		submanifolds}, Differential Geom. Appl. \textbf{4} (1994), no.~4, 391--403.
	
	\bibitem{Palmer}
	\bysame, \emph{Hamiltonian minimality and {H}amiltonian stability of {G}auss
		maps}, Differential Geom. Appl. \textbf{7} (1997), no.~1, 51--58.
	
\bibitem{QT}
Chao Qian and Zizhou Tang, \emph{Recent progress in isoparametric functions and
  isoparametric hypersurfaces}, Real and complex submanifolds, Springer Proc.
  Math. Stat., vol. 106, Springer, Tokyo, 2014, pp.~65--76. 

	\bibitem{Reckziegel}
	Helmut Reckziegel, \emph{Horizontal lifts of isometric immersions into the
		bundle space of a pseudo-{R}iemannian submersion}, Global differential
	geometry and global analysis 1984 ({B}erlin, 1984), Lecture Notes in Math.,
	vol. 1156, Springer, Berlin, 1985, pp.~264--279. 
	
	\bibitem{R1}
	\bysame, \emph{On the geometry of the complex quadric}, Geometry and topology
	of submanifolds, {VIII} ({B}russels, 1995/{N}ordfjordeid, 1995), World Sci.
	Publ., River Edge, NJ, 1996, pp.~302--315. 
	
	\bibitem{siffert}
	Anna Siffert, \emph{A new structural approach to isoparametric hypersurfaces in
		spheres}, Ann. Global Anal. Geom. \textbf{52} (2017), no.~4, 425--456.
	
	\bibitem{Smyth1967}
	Brian Smyth, \emph{Differential geometry of complex hypersurfaces}, Ann. of
	Math. (2) \textbf{85} (1967), 246--266. 
	
\bibitem{TY}
Zizhou Tang and Wenjiao Yan, \emph{Isoparametric theory and its applications}, arXiv:1709.07235.

	\bibitem{TU2015}
	Francisco Torralbo and Francisco Urbano, \emph{Minimal surfaces in
		{$\Bbb{S}^2\times\Bbb{S}^2$}}, J. Geom. Anal. \textbf{25} (2015), no.~2,
	1132--1156. 
	
\end{thebibliography}
\end{document}